\documentclass[3p,preprint]{elsarticle}

\usepackage{amsfonts}
\usepackage{amsmath}
\usepackage{geometry}
\usepackage{graphicx}
\usepackage{subcaption}
\usepackage{algorithmicx}
\usepackage{algorithm}
\usepackage{algpseudocode}
\usepackage{color}
\usepackage{relsize}
\usepackage{amssymb}
\usepackage{tikz}
\usepackage{csquotes}
\usepackage{multirow}

\newtheorem{theorem}{Theorem}[section]

\newtheorem{corollary}[theorem]{Corollary}

\newtheorem{example}[theorem]{Example}

\newtheorem{lemma}[theorem]{Lemma}

\newtheorem{proposition}[theorem]{Proposition}
\newtheorem{remark}[theorem]{Remark}

\newenvironment{proof}[1][Proof]{\noindent\textbf{#1.} }{\ \rule{0.5em}{0.5em}}

\biboptions{sort&compress}

\begin{document}

\bigskip

\bigskip 
\begin{frontmatter}

\title{$C^{\infty}$ rational approximation and quasi-histopolation of functions with jumps through multinode Shepard functions}

\author[address-CS]{Francesco Dell'Accio\corref{corrauthor}}
\cortext[corrauthor]{Corresponding author}
\ead{francesco.dellaccio@unical.it}

\author[address-CS]{Francesco Larosa}
\ead{francesco.larosa@unical.it}

\author[address-CS]{Federico Nudo}
\ead{federico.nudo@unical.it}

\author[address-MC]{Najoua Siar}
\ead{s.najoua@umi.ac.ma}

\address[address-CS]{Department of Mathematics and Computer Science, University of Calabria, Rende (CS), Italy}

\address[address-MC]{Department of Mathematics, University Moulay Ismail, Meknes, Morocco}

\begin{abstract}
Histopolation, or interpolation on segments, is a mathematical technique used to approximate a function $f$ over a given interval $I=[a,b]$ by exploiting integral information over a set of subintervals of $I$. Unlike classical polynomial interpolation, which is based on pointwise function evaluations, histopolation reconstructs a function using integral data. However, similar to classical polynomial interpolation, histopolation suffers from the well-known Runge phenomenon when integral data are based on a grid with many equispaced nodes, as well as the Gibbs phenomenon when approximating discontinuous functions. In contrast, quasi-histopolation is designed to relax the strict requirement of passing through all the given data points. This inherent flexibility can reduce the likelihood of oscillatory behavior using, for example, rational approximation operators. In this work, we introduce a $C^{\infty}$ rational quasi-histopolation operator, for bounded (integrable) functions, which reconstruct a function by defeating both the Runge and Gibbs phenomena. A key element of our approach is to blend local histopolation polynomials on a few nodes using multinode Shepard functions as blending functions. Several numerical experiments demonstrate the accuracy of our method. 
\end{abstract}


\end{frontmatter}

\section{Introduction}
The classical polynomial interpolation is a fundamental technique in numerical analysis and computational mathematics. It consists of approximating, by polynomials, an unknown real-valued function $f$ of one or several real variables by leveraging the values of $f$ at specific points. In the simplest univariate case, given the data
\begin{equation*}
    \left(x_i,y_i\right), \qquad x_i\in [a,b],\qquad  y_i=f\left(x_i\right), \qquad i=0,\dots, n,
\end{equation*}
 the goal of the classical polynomial interpolation is to find a polynomial of degree at most $n$, $p\in \mathbb{P}_n$,  such that
 \begin{equation*}
     p\left(x_i\right)=y_i, \qquad i=0,\dots,n. 
 \end{equation*}
 
The polynomial interpolation problem can be generalized to more abstract settings~\cite{Davis:1975:IAA}. For this, let $$V=\operatorname{span}\left\{ e_1,\dots,e_n \right\}$$ be a vector space of dimension $n$ and let $\left\{\mu_1,\dots,\mu_n\right\}$ be given linear functionals defined on $V$. In this scenario, given the data $\mu_i(f)\in\mathbb{R}$, the \textit{general problem of finite interpolation} aims to find an element $\mathcal{I}[f]\in V$ such that 
\begin{equation}\label{condint}
    \mu_i\left(\mathcal{I}[f]\right)=\mu_i(f), \qquad i=1,\dots,n. 
\end{equation}

The solution of the general problem of finite interpolation exists and it is unique if and only if the associated \textit{Gram matrix}
\begin{equation} \label{eq:gramian}
G = 
\begin{bmatrix}
\mu_{1}\left(e_1\right) & \mu_{1}\left(e_2\right) & \cdots & \mu_{1}\left(e_{n}\right)\\
\mu_{2}\left(e_1\right) & \mu_{2}\left(e_2\right) & \cdots & \mu_{2}\left(e_{n}\right)\\
\vdots  & \vdots  & \ddots & \vdots  \\
\mu_{n}\left(e_1\right) & \mu_{n}\left(e_2\right) & \cdots & \mu_{n}\left(e_{n}\right)\\
\end{bmatrix}
\end{equation}
is nonsingular. In this case, the approximation operator 
\begin{equation}\label{operatorI}
    \mathcal{I}:f {\longmapsto}\mathcal{I}[f] 
\end{equation}
is called an \textit{interpolation operator}. If $\mathbb{P}_k\subset V$ then, by the solution's uniqueness of the interpolation problem, it follows that $\mathcal{I}[p]=p$ for any $p\in \mathbb{P}_k$ and we say that $\mathcal{I}$ reproduces polynomials of degree less than or equal to $k$. 

An operator $\mathcal{Q_I}$ which uses the same data $\mu_i(f)$ and satisfies the requirement 
\begin{equation} \label{exactnesspol}
    \mathcal{Q_I}[p]=p, \qquad \forall p \in \mathbb{P}_k,
\end{equation}
without being an interpolation operator, is called a \textit{quasi-interpolation operator}~\cite{Wang:2004:QIW}. Quasi-interpolation operators then satisfy the less restrictive conditions 
\begin{equation}\label{propqi}
\mu_i\left(\mathcal{Q_I}[f]\right) \approx \mu_i\left(f\right), \qquad i=1,\dots,n,
\end{equation}
meaning that they only approximate the given data 
rather than requiring exact interpolation~\eqref{condint}. In practice, quasi-interpolation operators are designed to mitigate oscillations, especially in regions where traditional interpolants may exhibit the Runge phenomenon~\cite{Runge:1901:UEF}.
Moreover, in recent years, quasi-interpolation operators have also been used to approximate functions with jump discontinuities (or discontinuities of the first kind)~\cite{Arandiga1, Arandiga:2024:EAW,Arandiga:2024:NUW, Arandiga2}. 
In this hypothesis, classical interpolation methods often struggle to effectively manage discontinuities, leading to the Gibbs phenomenon~\cite{Gibbs:1898:FSS}. 

When the linear functionals are definite integrals over segments $s_i \subset \mathbb{R}$, i.e.,
\begin{equation}\label{intfuncs}
    \mu_i(f)=\int_{s_i}f(x) dx, \qquad i=1,\dots,n, \qquad \mathcal{S}=\{s_1,\dots,s_n\},
\end{equation}
and $V=\mathbb{P}_{n-1}$, the interpolation problem is also referred to as \textit{histopolation problem} and the polynomial $p \in \mathbb{P}_{n-1}$ satisfying~\eqref{condint} is said an \textit{histopolant} on the set $\mathcal{S}$. 
In this case, we set $\mathcal{H}(f)=\mathcal{I}(f)$ and we call the operator 
\begin{equation*}
    \mathcal{H}:f {\longmapsto}\mathcal{H}[f] 
\end{equation*}
an \textit{histopolation operator}. An operator $\mathcal{Q_{H}}$ satisfying~\eqref{exactnesspol} and~\eqref{propqi} is referred to as a \textit{quasi-histopolation operator}. 
Unlike classical univariate polynomial interpolation, $n$ pairwise distinct histopolation segments do not assure the uniqueness of the histopolation polynomial. A set of segments $\mathcal{S}=\left\{s_1,\dots,s_n\right\}$ is said to be \textit{unisolvent} if the polynomial histopolant on this set is unique. 
This happens, for example, if the segments $s_i$ are bounded by pairs $x_{i-1}, x_{i}$ of pairwise distinct points $x_0<x_1<\dots<x_n$. For more details see~\cite{Bruno:2023:PIO}.    

Histopolation has found several practical applications across various fields, including splines~\cite{Schoenberg:1973:SFA, Fischer:2005:MPR, Fischer:2007:CSP, Siewer:2008:HIS, Hallik:2017:QLR}, fractal functions~\cite{Barnsley:2023:HFF}, preconditioning~\cite{Hiptmair:2007:NAS}, image processing~\cite{Bosner:2020:AOC}, etc.
Unlike classical polynomial interpolation, which requires continuity, histopolation only assumes that the function is essentially bounded, making it applicable to a broader class of functions.
However, like classical polynomial interpolation, histopolation can suffer from the Runge phenomenon when integral data are based on a grid with many equispaced nodes (the ends of the segments), as well as the Gibbs phenomenon when approximating discontinuous functions~\cite{Bruno:2023:PIO}. To address these challenges, in the classical polynomial interpolation, rational approximation operators have been proposed~\cite{Petrushev:2011:RAO, DellAccio:2016:APT, DellAccio:2019:RCM, 
DellAccio:2020:HSM}. The construction of rational approximation operators by blending local histopolants using classical Shepard functions is also discussed in~\cite{Demichelis:1995:GAO}.

 The main goal of this work is to introduce a $C^{\infty}$ rational quasi-histopolation operator for bounded (integrable) functions, which is able to reconstruct a function while overcoming the Runge and Gibbs phenomena, by employing multinode Shepard functions. Among the assumptions of work, in case of discontinuous functions, we assume that the number of discontinuities is finite and that the discontinuity points have been localized by one of the standard methods, such as using the smoothness indicators, see~\cite{Canny:1986:ACA}, which detect whether a node is near a discontinuity.

The structure of the paper is as follows. Section~\ref{sec2} introduces a $C^{\infty}$ rational quasi-histopolation operator constructed by combining multinode Shepard functions with local histopolants. Furthermore, the key properties of the multinode Shepard functions are recalled. In Section~\ref{sec3}, leveraging these properties, we investigate an error bound associated with the proposed quasi-histopolation operator in the case of equispaced nodes.  Section~\ref{sec4} presents numerical results that demonstrate the accuracy of the proposed method. Finally, conclusions and future perspectives of work are provided in Section~\ref{conclusions}.

\section{Rational quasi-histopolation operators}\label{sec2}
In order to ensure the unisolvence of all local polynomial histopolants  considered below, we assume that $X_n=\left\{x_i \,:\, i=0,\dots,n \right\}$ is a set of $n+1$ nodes in $[a,b]$ such that 
 \begin{equation*}
     a=x_0<x_1<\dots<x_{n-1}<x_n=b.
 \end{equation*}

 We denote by $f$ an unknown real-valued, not necessarily continuous, function defined on $[a,b]$. 
 We aim to reconstruct the function $f$ by assuming to know the following data
 \begin{equation}\label{intefun}
     \mu_i(f)=\int_{s_i} f(x) dx, \qquad i=1,\dots,n,
 \end{equation}
where 
\begin{equation*}
    s_i=\left[x_{i-1},x_i\right], \quad i=1,\dots,n,
\end{equation*}
and
\begin{equation}\label{neweq}
\mathcal{S}_n=\left\{s_i \,:\, i=1,\dots,n \right\},
\end{equation}
is the \textit{set of segments based on the set} of nodes $X_n$.
To this end, we introduce a $C^{\infty}$ rational quasi-histopolation operator using multinode Shepard functions and local histopolants.

In order to define this approximation operator, some settings are required. By letting $m<n$, we define 
\begin{equation*}
    Y_m=\left\{y_j \,:\, j=0,\dots,m\right\}
\end{equation*}
as the set of points where the function $f$ is discontinuous, assuming that
\begin{equation*}   y_j\in\mathring{s}_{\sigma(j)}=\left(x_{\sigma(j)-1},x_{\sigma(j)}\right), \quad j=0,\dots,m,
\end{equation*}
where
\begin{equation*}
    \sigma: \{0,\dots,m\} \rightarrow \{1,\dots,n\}
\end{equation*}
is a strictly increasing function and $\mathcal{S}^{\prime}_m=\left\{ s_{\sigma(j)} \, : \, j=0,\dots,m \right\} \subset \mathcal{S}_n$. 
We introduce the $m+2$ sets
\begin{equation}\label{intervalI}
    I_0=\bigcup_{k=1}^{\sigma
    (0)-1} s_{k}, \quad I_{\ell}=\bigcup_{k=\sigma
    (\ell-1)+1}^{\sigma(\ell)-1} s_{k}, \quad
    \ell=1,\dots,m, \quad I_{m+1}= \bigcup_{k=\sigma
    (m)+1}^{n} s_{k},
\end{equation}
and notice that, for any $\ell\in\left\{0,\dots,m+1\right\}$, only one of the following holds
\begin{itemize}
    \item $I_{\ell}=\emptyset$,
    \item $I_{\ell}$ is an interval without singularities of $f$.
\end{itemize}
In the last case, we call $I_{\ell}$ \textit{interval of continuity} of $f$, we denote its length by $\left\lvert I_{\ell} \right\rvert$ and we set
\begin{equation*}
 h_{\ell}=\left\{ \begin{array}{cl}
   0 & \text{if }  I_{\ell}=\emptyset, \\[10pt]
  \left\lvert I_{\ell} \right\rvert & \text{if } I_{\ell}\neq \emptyset.
\end{array}\right.
\end{equation*}
We also set
\begin{equation*}
h^{\mathrm{min}}=\min\limits_{\ell=0,\dots,{m+1}} \left\{  h_{\ell} \,:\, h_{\ell} \neq 0  \right\}, 
\end{equation*}

\begin{equation}\label{hsn}
h^{\min}_{\mathcal{S}_n}=\min\limits_{i=1,\dots,n} \left\{ \left\lvert s_i\right\rvert \,:\, s_i \notin \mathcal{S}^{\prime}_m \right\}, \qquad h^{\max}_{\mathcal{S}_n}=\max\limits_{i=1,\dots,n} \left\{ \left\lvert s_i\right\rvert \,:\, s_i \notin \mathcal{S}^{\prime}_m \right\}
\end{equation}
and assume that 
\begin{equation}\label{hi}
   h^{\mathrm{min}} \geq  h^{\max}_{\mathcal{S}_n}. 
\end{equation}

\begin{example}
    Let assume that $a=0$, $b=1$, $x_i=\frac{i}{10}$, $i=0,1,\dots,10$, and $y_0=\frac{1}{20}$, $y_1=\frac{7}{20}$, $y_2=\frac{17}{20}$. Consequently  $\sigma(0)=1$, $\sigma(1)=4$, $ \sigma(2)=9$, and
    \begin{equation*}
        I_0=\emptyset, \quad I_1=\left[x_1,x_3\right], \quad I_2=\left[x_4, x_8\right], \quad I_3=[x_9,x_{10}].
    \end{equation*}
\end{example}

% \begin{remark}
%  We notice that if
%  \begin{equation*}
%      \sigma(k+1)=\sigma(k)+1 \text{ for some } k=0,\dots,m-1,
%  \end{equation*}
%  then $y_k \in s_{\sigma(k)}$ and $y_{k+1} \in s_{\sigma(k+1)}$. In this case, by definition~\eqref{intervalI}, it results $I_{k+1}=\emptyset$.
% \end{remark}
\begin{remark}
If the function $f$ is continuous, the set $Y_m = \emptyset$. In this case, we set 
\begin{equation*}
    I_0=\bigcup_{i=1}^n s_i =[a,b].
\end{equation*}
In other words, we have only one interval of continuity. 
\end{remark}
To introduce the quasi-histopolation operator under the condition~\eqref{hi}, we let $d\in\mathbb{N}_0$ denote a non-negative integer. The construction is based on the existence of a finite covering of each interval of continuity $I_{\ell}$ by intervals $I\subset I_{\ell}$, each of which contains at least $d+1$ segments $s_i \subset I$ from the dataset $\{(s_i, \mu_i(f))\}_{i=1}^n$, that allows the introduction of a histopolation polynomial of degree at least $d$. To this aim, for any two intervals $V_1,V_2 \subset [a,b]$ we set 
\begin{equation*}
\Psi\left(V_1,V_2\right)=\begin{cases}
        1, & \text{if } V_1 \subset V_2\\
        0, &  \text{otherwise.}
    \end{cases}
\end{equation*}
Moreover, for any $r>0$ and for any $ x \in  \bigcup\limits_{\ell=0}^{m+1} I_{\ell}$, we set  
\begin{equation*}
    \mathcal{I}_{r,x}=\left\{I\subset \bigcup_{\ell=0}^{m+1} I_{\ell} \, :  \, \text{$I$ is a closed interval,} \, \lvert I \rvert = r, \, \exists s \in \mathcal{S}_n \, \text{ s.t. } \, x\in s \subset I \right\},
\end{equation*}
and for any $d\in\mathbb{N}_0$ we also set
\begin{equation}\label{rd}
    r_d=\inf \left\{r\in\mathbb{R}_{+}\, :\, ( r\leq h^{\mathrm{min}}) \wedge \left(\forall x \in  \bigcup_{\ell=0}^{m+1} I_{\ell} \quad \exists \,  I \in \mathcal{I}_{r,x} \, \text{ s.t.}  \,  \sum_{i=1}^n \Psi(s_i,I)\geq d+1\right)\right\}.
\end{equation}

% by $r_d \in \mathbb{R}_+ $ the smallest positive real number $r$, if it exists, such that
% \begin{equation*} 
% \textcolor{red}{(} r\leq h^{\mathrm{min}}\textcolor{red}{)} \wedge \left(\textcolor{red}{\forall} x \in  \bigcup_{\ell=0}^{\textcolor{red}{m+1}} I_{\ell} \quad \textcolor{red}{\exists} \,  I \in \mathcal{I}_{r,x} \, \textcolor{red}{ s.t. } \,  \sum_{i=1}^n \Psi(s_i,I)\geq d+1\right).
%     \end{equation*}
% Then we have 
\begin{theorem}\label{theorem:hxnmax}
Under the assumption~\eqref{hi}, $r_0=h^{\max}_{\mathcal{S}_n}$.
% the following result holds: 
%  \begin{equation*}
%  r_0=\inf \left\{r\in\mathbb{R}_{+}\, :\, \textcolor{red}{(} r\leq h^{\mathrm{min}}\textcolor{red}{)} \wedge \left(\textcolor{red}{\forall} x \in  \bigcup_{\ell=0}^{\textcolor{red}{m+1}} I_{\ell} \quad \textcolor{red}{\exists} \,  I \in \mathcal{I}_{r,x} \, \textcolor{red}{ s.t. } \,  \sum_{i=1}^n \Psi(s_i,I)\geq 1\right)\right\}  
% \end{equation*}
% exists.
\end{theorem}
\begin{proof}
    Let $\ell\in\{0,\dots,m+1\}$ be such that 
$    I_\ell= \left[a_\ell, b_\ell\right]\neq \emptyset.$
For each $x \in I_\ell$, if $x \notin X_n$ we denote by $s(x)$ the unique segment in $\mathcal{S}_n$ that contains $x$; 
if instead $x \in X_n$, then $x$ lies at the boundary of at most two segments in $\mathcal{S}_n$, which we denote by $s_-(x)$ (the segment for which $x$ is the right endpoint) and
$s_+(x)$ (the segment for which $x$ is the left endpoint). To prove the theorem, we show that it is possible to construct a sequence of intervals $\mathcal{U}_{\ell}=\left\{U_{\ell,1}\dots,U_{\ell,n_{\ell}}\right\}$ such that, for any $j=1,\dots,n_{\ell}$
\begin{equation}\label{lengthandsum}
    \left \lvert U_{\ell,j} \right \rvert = h^{\max}_{\mathcal{S}_n}, \qquad \sum_{i=1}^n \Psi \left( s_i, U_{\ell,j}\right) \geq 1,    
\end{equation}
and, for any $x \in I_{\ell}$ 
\begin{equation}\label{xinsx}
    x \in s \left(x\right) \subset U_{\ell,j}\subset I_{\ell},
\end{equation}
for some $j=1,\dots,n_{\ell}$. The last property ensures that
\begin{equation}
\bigcup_{j=1}^{n_{\ell}} U_{\ell,j}=I_{\ell}.
\end{equation}
 From $h^\mathrm{min}\ge h^{\max}_{\mathcal{S}_n}\ge \left\lvert s_+\left(a_\ell\right) \right\rvert$ it results
        \begin{equation}
            a_\ell\in s_+\left(a_{\ell}\right)\subset \left[a_\ell,a_\ell+h^{\max}_{\mathcal{S}_n}\right]\subset I_\ell
        \end{equation}
        and we set 
        \begin{equation}
U_{\ell,1}=\left[a_{\ell,1},b_{\ell,1}\right]=\left[a_\ell,a_\ell+h^{\max}_{\mathcal{S}_n}\right].
        \end{equation}
Then 
\begin{equation*}
     \left \lvert U_{\ell,1} \right \rvert = h^{\max}_{\mathcal{S}_n}, \qquad \sum_{i=1}^n \Psi \left( s_i, U_{\ell,1} \right) \geq 1,
\end{equation*}
so that, if $b_{\ell,1}=b_{\ell}$, we have got the sequence $\mathcal{U}_{\ell}$, since the property~\eqref{xinsx} is automatically ensured. Otherwise, we set
\begin{equation*}
    s=\begin{cases}    s_+\left(b_{\ell,1}\right), &\text{ if } b_{\ell,1} \in X_n, \\
       s(b_{\ell,1}) & \text{ if } b_{\ell,1} \notin X_n,
    \end{cases}
\end{equation*}
and either $\min s+h^{\max}_{\mathcal{S}_n}> b_\ell$ or $\min s+h^{\max}_{\mathcal{S}_n}\leq b_\ell$.\\
If $\min s +h^{\max}_{\mathcal{S}_n}> b_\ell$, we set 
\begin{equation*}
U_{\ell,2}=\left[a_{\ell,2},b_{\ell,2}\right]=\left[b_{\ell}-h_{\mathcal{S}_n}^{\max},b_{\ell}\right]. 
\end{equation*}
Then
\begin{equation*}
     \left \lvert U_{\ell,2} \right \rvert = h^{\max}_{\mathcal{S}_n}, \qquad \sum_{i=1}^n \Psi \left( s_i, U_{\ell,2} \right) \geq 1
\end{equation*}
since
\begin{equation}
b_\ell\in s_-\left(b_{\ell}\right)\subset \left[b_\ell-h^{\max}_{\mathcal{S}_n},b_\ell\right]\subset I_\ell
\end{equation}
and 
\begin{equation*}
   U_{\ell,1}\cap U_{\ell,2} \neq \emptyset 
\end{equation*}
since $a_{\ell,2}=b_{\ell}- h^{\max}_{\mathcal{S}_n} < \min s \leq b_{\ell,1} < b_{\ell}=b_{\ell,2}$. 
By construction, for any $x \in I_{\ell}=U_{\ell,1} \cup U_{\ell,2}$ 
\begin{equation}
   x\in s(x) \subset U_{\ell,1} \quad \text{or} \quad x \in s(x)\subset U_{\ell,2},
\end{equation}
and then we have got the sequence $\mathcal{U}_{\ell}$.\\
If, instead, $\min s +h^{\max}_{\mathcal{S}_n}\leq b_\ell$, we set 
\begin{equation*}
U_{\ell,2}=\left[a_{\ell,2},b_{\ell,2}\right]=\left[\min s,\min s+h_{\mathcal{S}_n}^{\max}\right].
\end{equation*}
Then 
\begin{equation*}
    \left\lvert U_{\ell,2} \right\rvert = h^{\max}_{\mathcal{S}_n}, \qquad \sum\limits_{i=0}^n \Psi\left(s_i,U_{\ell,2}\right)\geq1,
\end{equation*}
and by setting
      \begin{equation*}
  s^{\prime} =  \begin{cases}
        s_{-}\left(\min s + h_{\mathcal{S}_n}^{\max}\right) & \text{if } \min s + h_{\mathcal{S}_n}^{\max} \in X_n,\\
        s \left(\min s + h_{\mathcal{S}_n}^{\max}\right) & \text{if } \min s + h_{\mathcal{S}_n}^{\max} \notin X_n,
   \end{cases}     
   \end{equation*}
   we notice that, for any $x \in \left[a_{\ell,1}, \min s\right]$ it results 
   \begin{equation*}
       x \in s(x) \subset U_{\ell,1},
   \end{equation*}
    while, for any $x \in \left[\min s, \min s^{\prime}\right]$ it results  
   \begin{equation*}
       x \in s(x) \subset U_{\ell,2}.
   \end{equation*} 
   Moreover, $U_{\ell,1}\cap U_{\ell,2} \neq \emptyset$. If $b_{\ell,2}= b_{\ell}$, we have got the sequence $\mathcal{U}_{\ell}$. Otherwise, we assume that we have already got a sequence of intervals $U_{\ell,1},\dots, U_{\ell,k-1}$, $k \geq 3$, such that
   \begin{equation*}
       \left \lvert U_{\ell,j} \right \rvert = h^{\max}_{\mathcal{S}_n}, \qquad  \sum\limits_{i=0}^n \Psi\left(s_i,U_{\ell,j}\right)\geq1,
   \end{equation*}
for any $j=1,\dots,k-1$, and, by setting
   \begin{equation*}
  s^{\prime} = \begin{cases}
         s_- \left(b_{\ell,k-1}\right) & \text{if } b_{\ell,k-1} \in X_n,\\
        s \left(b_{\ell,k-1}\right) & \text{if } b_{\ell,k-1} \notin X_n,
   \end{cases}     
   \end{equation*}
  for any $x \in \left[ a_{\ell,1}, \min s^{\prime} \right]$, it results   
   \begin{equation*}
       x \in s(x) \subset U_{\ell,j},
   \end{equation*}
     for some $j=1,\dots,k-1$. The previous step shows that it is possible to construct another interval of the sequence, which will either allow us to obtain the sequence $U_{\ell}$ or require another step. The process certainly ends since $\mathcal{S}_n$ is finite.\\ 
The existence of the sequence $\mathcal{U}_{\ell}$ satisfying~\eqref{lengthandsum} and~\eqref{xinsx} show that $h^{\max}_{\mathcal{S}_n}$ satisfies the condition
    \begin{equation*}
        ( h^{\max}_{\mathcal{S}_n}\leq h^{\mathrm{min}}) \wedge \left(\forall x \in  I_{\ell} \quad \exists \,  I \in \mathcal{I}_{h^{\max}_{\mathcal{S}_n},x} \,  \text{ s.t. }  \,  \sum_{i=1}^n \Psi(s_i,I)\geq 1\right).
    \end{equation*}
    Finally, let  $s^{\max}=\left[x_{i_{\max}},x_{i_{\max}+1}\right]$ be an interval of length $h_{\mathcal{S}_n}^{\max}$ and 
$x_{\operatorname{middle}}=\left(x_{i_{\max}+1}+x_{i_{\max}}\right)/2$ its middle point. Since, for any positive number $r<h^{\max}_{\mathcal{S}_n}$, the set $\mathcal{I}_{r,x_{\operatorname{middle}}}$ is the empty set, the thesis follows.
\end{proof}

Let ${D}_i=\left[x_i,x_{i+d+1}\right]$, $i=0,\dots, n-d-1$. In analogy with the previous theorem, it is possible to prove the following result. 
\begin{theorem} \label{theoremrd}
 Under the assumption
\begin{equation*}
h^{\min} \geq h^{\max,d}_{\mathcal{S}_n}=\max_{\ell=0,\dots,m+1} \max_{\substack{i = 0, \dots, n-d-1 \\ D_i \subset I_\ell}}\left\lvert {D}_i\right\rvert, 
\end{equation*} 
 $r_d$ exists and it is strictly positive.
% $$r_d=\max_{\ell=0,\dots,m+1} \max_{\substack{i = 0, \dots, n-d-1 \\ D_i \subset I_\ell}}\left\lvert {D}_i\right\rvert.$$
\end{theorem}
\begin{proof}
We use the same notation as that employed in the proof of Theorem~\ref{theorem:hxnmax}. To prove the existence of $r_d>0$, we show that it is possible to construct a sequence of intervals $U_{\ell,1}, \dots, U_{\ell,n_{\ell}}$ such that, for any $j=1,\dots,n_{\ell}$,
 \begin{equation*}
       \left \lvert U_{\ell,j} \right \rvert = h^{\max,d}_{\mathcal{S}_n}, \qquad \sum_{i=1}^n \Psi\left(s_i,U_{\ell,j}\right)\geq d+1,
   \end{equation*}
   and for any $x \in I_{\ell}$
   \begin{equation*}
       x \in s\left(x\right) \subset U_{\ell,j} \subset I_{\ell},
   \end{equation*}
   for some $j=1,\dots,n_{\ell}$.
The argument follows similar lines to those used in the proof of Theorem~\ref{theorem:hxnmax}, and relies on the fact that any interval of length $h^{\max,d}_{\mathcal{S}_n}$ must contain at least $d+1$ segments of $\mathcal{S}_n \setminus \mathcal{S}^{\prime}_m$. 
\end{proof}

\begin{remark}
    Since we have only $n$ segments, $r_{n+1}$ does not exist. 
\end{remark}
In the following, we set
\begin{equation} \label{dmax}
  d_{\max}=\max\left\{ d\in\mathbb{N} \,:\, r_d\text{ exists} \right\}.  
\end{equation}
\begin{remark}
    We notice that, by definition, if $r_{d_{\max}}$ exists, then $r_d$ also exists for any $d \in \{0, \dots, d_{\max}\}$.
\end{remark}

For any $\ell=0,\dots,m+1$ such that $I_{\ell}\neq\emptyset$ we set
\begin{equation}\label{Felle}
      \mathcal{F}_{\ell} =  \left\{ I \in \mathcal{I}_{r_d,x} \,  \text{ s.t. }  \,  \sum_{i=1}^n \Psi(s_i,I)\geq d+1 \right\}_{x \in I_{\ell}}
    \end{equation}
    and 
    \begin{equation}
        \mathcal{U}_{\ell} = \left\{ U_{\ell,k}=\left[a_{\ell,k},b_{\ell,k}\right] \in \mathcal{F}_{\ell} \, :\, k=1,\dots,n_{\ell} \right\},
    \end{equation}
    where $U_{\ell,k}$ has been introduced in the proof of Theorem~\ref{theorem:hxnmax} and, by analogy, in the proof of Theorem~\ref{theoremrd}.
We notice that the family $\mathcal{U}_{\ell}$ is a covering of $I_{\ell}$ by subset of $\mathcal{F}_{\ell}$ satisfying the following properties
\begin{itemize}
    \item[(C)] $\forall x \in I_{\ell}, \quad \exists k \in \left\{1, \dots, n_{\ell}\right\}$ such that $x \in s(x) \subset U_{\ell,k}$;
    \item[(M)] $\forall k \in \left\{1, \dots, n_{\ell}\right\}, \quad \exists x \in I_{\ell} $ such that $x\in s(x) \not\subset \bigcup\limits_{\substack{j = 1 \\ j \ne k}}^{n_{\ell}} U_{\ell,j}$.
\end{itemize}
\begin{remark}
Property (C) ensures that every segment $s_i \in S_n$ is fully contained in some interval $U_{\ell,j}$, for suitable $\ell = 0, \ldots, m+1$ and $j = 1, \ldots, n_\ell$. Property (C) is also guaranteed if, in equation~\eqref{Felle}, $r_d$ is substituted by $h^{\max,d}_{\mathcal{S}_n}$. Property (M), on the other hand, guarantees that for each $\ell$ such that $I_\ell \neq \emptyset$, the family $\mathcal{U}_\ell$ satisfying (C) is minimal. 
Property (C) is necessary for the objective of this work, while property (M) is only desirable and not necessarily satisfied if, in equation~\eqref{Felle}, $r_d$ is substituted by $h^{\max,d}_{\mathcal{S}_n}$.
\end{remark}

For any $\ell = 0,\dots,m+1$ such that $I_{\ell}\neq \emptyset$ we cover the interval of continuity $I_{\ell}$ by using subset of intervals $\left\{ U_{\ell,j} \right\}_{j=1}^{n_{\ell}}$, for which we have 
\begin{itemize}
\item[-] $I_{\ell}= \bigcup\limits_{j=1}^{n_{\ell}} U_{\ell,j}$,
\item[-] $\left \lvert U_{\ell,j} \right \rvert = r_d$, $j=1,\dots,n_{\ell}$,
\item[-] $U_{\ell,j} \cap U_{\ell,j-1}\neq \emptyset$, $j=2,\dots,n_{\ell}$,
% $\mathrm{card}\left(U_{\ell,j} \cap U_{\ell,j-1}\right) \textcolor{red}{\geq} 1$, $j=1,\dots,n_{\ell}-1$, and $\mathrm{card}\left(U_{\ell,n_{\ell}} \cap U_{\ell,n_{\ell}-1}\right) \geq 1$,
\end{itemize}
where $n_{\ell}$ is uniquely determined by the minimality condition (M) of the covering.
% \begin{equation*}
% n_{\ell}=\left\{ \begin{array}{lr}
%    \left\lfloor\dfrac{h_{\ell}}{r_d}\right\rfloor-1, & \text{if } \left\lfloor\dfrac{h_{\ell}}{r_d}\right\rfloor=\dfrac{h_{\ell}}{r_d}, \\[10pt]
%   \left\lfloor\dfrac{h_{\ell}}{r_d}\right\rfloor , & \text{if }
%     \left\lfloor\dfrac{h_{\ell}}{r_d}\right\rfloor<\dfrac{h_{\ell}}{r_d}, 
% \end{array}\right.
% \end{equation*}
% with $\left\lfloor \cdot \right\rfloor$ the integer part of the real argument and $\mathrm{card}(\cdot)$ is the cardinality operator. 
Let us denote by $\mathrm{card}(\cdot)$ the cardinality operator. Let us consider 
\begin{equation}\label{setSiota}
\mathcal{S}_{\ell,j}=\left\{s_i \in \mathcal{S}_n \,:\,i=1,\dots,n, \, s_i \subset U_{\ell,j}\right\}, \quad \ell=0, \dots, m+1, \quad j=1,\dots,n_{\ell},
\end{equation}
and for any $\ell=0,\dots,m+1$ and $j=1,\dots,n_{\ell}$, we denote by
\begin{equation}\label{histopolant}
k_{\ell,j}=\mathrm{card}\left(\mathcal{S}_{\ell,j}\right),\qquad 
p_{\ell,j}(x)=p_{\ell,j}(f,x)\in \mathbb{P}_{k_{\ell,j}-1},  
\end{equation}
the histopolation polynomial of the function $f$ on the set of segments $\mathcal{S}_{\ell,j}$. For additional details on the histopolation polynomial, including existence, computational cost and stability properties, we refer the reader to~\cite{Bruno:2023:PIO}.
In some cases below, for simplicity, we represent the pair of indices $(\ell,j)$ using a single index $\iota$, that is
\begin{equation} \label{notation}
    \mathcal{S}_{\ell,j}\rightarrow \mathcal{S}_{\iota}, \quad U_{\ell,j}\rightarrow U_{\iota}, 
    \quad p_{\ell,j}\rightarrow p_{\iota},
    \quad\iota=1,\dots,M, \quad M=\sum\limits_{\ell=0}^{m+1} \left(n_{\ell}+1\right).
\end{equation}
% For any $K\in \mathbb{N}$ and any \textcolor{red}{$\iota=1,\dots,M$, we introduce sets of uniformly distributed points,}
% \textcolor{red}{
% \begin{equation}\label{Ciota}
%     C_{\iota}=\left\{\xi^{\iota}_{\kappa} \, : \,  a_{\iota}<\xi^{\iota}_{1}<\dots<\xi^{\iota}_{K}<b_{\iota}\right\}, \quad a_{\iota}= \min U_{\iota}, \ b_{\iota}=\max U_{\iota},
% \end{equation}
% not necessarily equispaced, such that the sets  $C_\iota$ and $C_{\lambda}$ share the same points on the overlap $U_\iota \cap U_{\lambda}$. The multinode Shepard functions based on the family of the sets $\left\{C_{\iota}\right\}^M_{\iota=1}$ are defined for each $\mu>0$ as follows
% \begin{equation}\label{multinodefunction}
% W_{\mu,\iota}(x)=\frac{\prod\limits_{\kappa=1}^{K} \left\lvert x-\xi^{\iota}_{\kappa} \right\rvert^{-\mu}}{\sum\limits_{\lambda=1}^M\prod\limits_{\kappa=1}^{K} \left\lvert x-\xi^{\lambda}_{\kappa} \right\rvert^{-\mu}}, \qquad \iota=1,\dots,M, \quad x \in [a,b].
% \end{equation}
% }
For any $K\in \mathbb{N}$ and any $\iota=1,\dots,M$, we introduce sets of pairwise distinct points,

\begin{equation}\label{Ciota}
    C_{\iota}=\left\{\xi^{\iota}_{\kappa} \, : \,  a_{\iota}<\xi^{\iota}_{1}<\dots<\xi^{\iota}_{K}<b_{\iota}\right\}, \quad a_{\iota}= \min U_{\iota}, \ b_{\iota}=\max U_{\iota},
\end{equation}
not necessarily equispaced. The multinode Shepard functions based on the sets $\left\{C_{\iota}\right\}^M_{\iota=1}$ are defined for each $\mu>0$ as follows
\begin{equation}\label{multinodefunction}
W_{\mu,\iota}(x)=\frac{\prod\limits_{\kappa=1}^{K} \left\lvert x-\xi^{\iota}_{\kappa} \right\rvert^{-\mu}}{\sum\limits_{\lambda=1}^M\prod\limits_{\kappa=1}^{K} \left\lvert x-\xi^{\lambda}_{\kappa} \right\rvert^{-\mu}}, \qquad \iota=1,\dots,M, \quad x \in [a,b].
\end{equation}

\begin{remark}\label{multinodeproprerties}
The multinode Shepard functions were first introduced in~\cite{dell2016multinode, DellAccio:2018:ROF} to construct global rational interpolants that solve Birkhoff interpolation problems not otherwise solvable by classical approaches, even within appropriate polynomial or rational spaces. The idea is to decompose the initial problem into subproblems, each with a unique local polynomial solution, and then to use multinode rational basis functions, based on the local interpolation points, to obtain a global interpolant by gluing together the local approximants. 
\end{remark}
The multinode Shepard functions are non-negative and form a partition of unity, that is 
\begin{eqnarray}\label{nonnegprop}
   W_{\mu ,\iota}\left(x\right) \geq 0, \\
   \sum\limits_{\iota=1}^M W_{\mu ,\iota}\left(x\right)=1.
   \label{partofuni}
\end{eqnarray}
 Moreover, $W_{\mu,\iota}(x)$ vanishes at all points $\xi^{\lambda}_{\kappa}$ that are not in $C_{\iota}$, that is 
\begin{equation}\label{vanishes}
    W_{\mu,\iota}(\xi^{\lambda}_{\kappa})=0, \qquad \xi^{\lambda}_{\kappa} \notin C_{\iota}.
\end{equation}
By properties~\eqref{partofuni} and~\eqref{vanishes} easily follows that $W_{\mu,\iota}$ and $W_{\mu,\lambda}$ could present wild oscillations in $U_{\iota}\cap U_{\lambda}$, $\iota \neq \lambda$, since the presence of the points $\xi^{\lambda}_{\kappa} \notin C_{\iota}$ and $\xi^{\iota}_{\kappa} \notin C_{\lambda}$. These oscillations can be avoided by forcing the sets $C_\iota$ and $C_{\lambda}$ to share the same points on the overlap $U_\iota \cap U_{\lambda}$.  
\begin{example}
In Figure~\ref{plot1606} the interval $[-1,1]$ is covered by the following sets
$U_{\iota}=\left[a_{\iota},b_{\iota}\right]$, $\iota=1,\dots,9$,
\begin{equation*}
    \begin{array}{lll}
U_1 = [-1.0000,\ -0.6333], & U_2 = [-0.8032,\ -0.4365], & U_3 = [-0.6233,\ -0.2566], \\
U_4 = [-0.3924,\ -0.0257], & U_5 = [-0.1773,\ +0.1894], & U_6 = [+0.0408,\ +0.4075], \\
U_7 = [+0.3290,\ +0.6957], & U_8 = [+0.5646,\ +0.9312], & U_9 = [+0.6333,\ +1.0000],
\end{array}
\end{equation*}
each of length $\frac{11}{30}$.
The left bounds $a_{\iota}$ are represented by green stars, while the right bounds $b_{\iota}$ are represented by red stars. These points, organized in an increasing sequence, subdivide the interval $[-1,1]$ in open intervals $O_{\ell}$.  
% In this particular case, $U_{\iota}$ intersects $U_{\iota+1}$, $\iota=1,\dots,6$, and $U_{7}\cap U_{8} \cap U_{9} \neq \emptyset$. 
Each $C_{\iota}$ contains exactly $10$ points $\xi^{\iota}_{\kappa}$, represented by yellow circles. In the Figure~\ref{plot1606A} the points $\xi^{\iota}_{\kappa}$ are equispaced in each open interval $U_{\iota}$ and wild oscillations of the multinode Shepard functions $W_{\mu,\iota}$ appear. In the Figure~\ref{plot1606B}, the points $\xi^{\iota}_{\kappa}$ in each open interval $O_{\ell}$ are equispaced in $O_{\ell}$, leading to well-behaved multinode Shepard functions $W_{\mu,\iota}$.
\end{example}
\begin{figure}[h!]
    \centering
    \begin{subfigure}[b]{0.49\linewidth}
        \includegraphics[width=\linewidth]{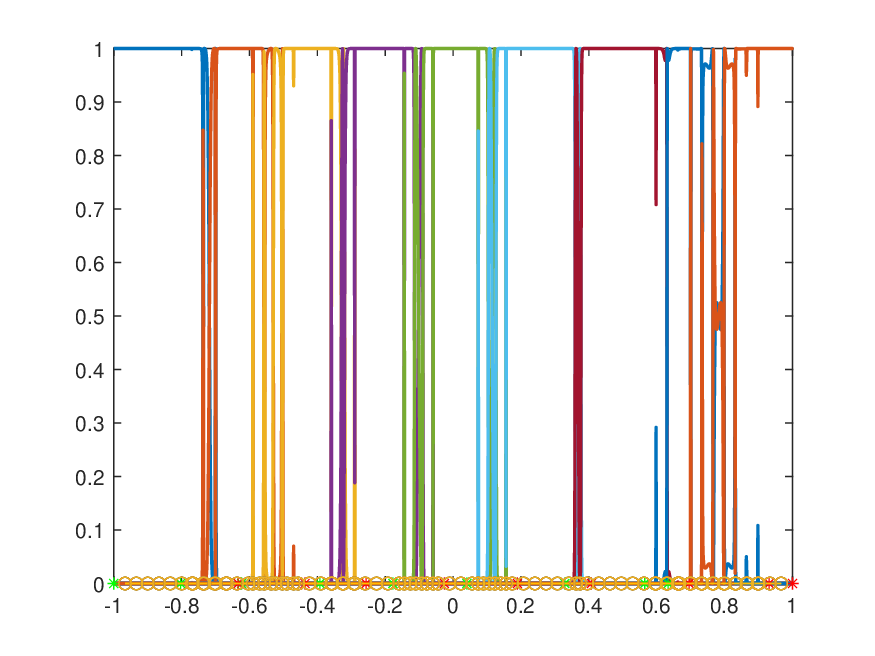}
        \caption{}
        \label{plot1606A}
    \end{subfigure}
    \hfill
    \begin{subfigure}[b]{0.49\linewidth}
        \includegraphics[width=\linewidth]{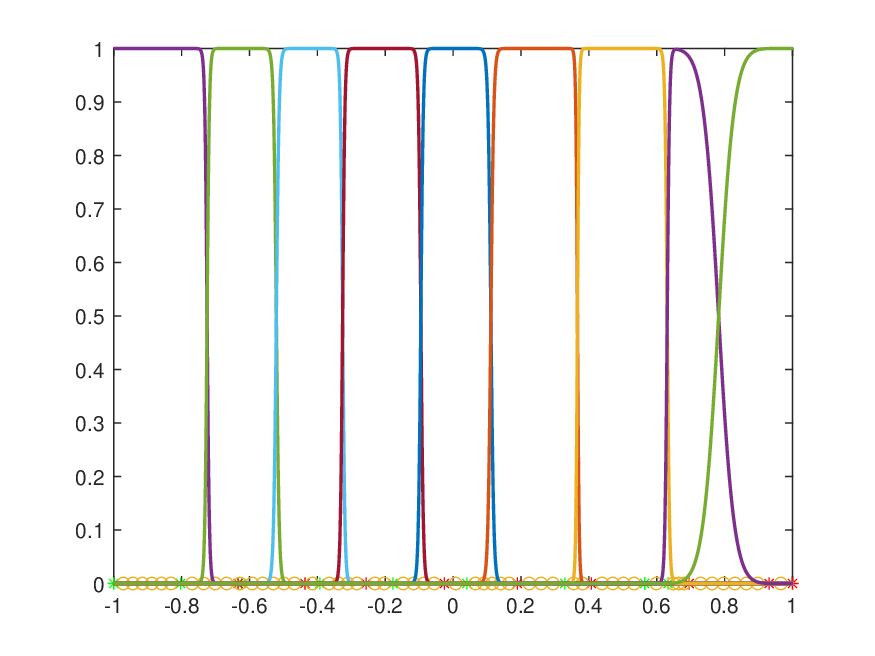}
        \caption{}
        \label{plot1606B}
    \end{subfigure}
    \caption{Plot of the multinode Shepard functions $W_{4,\iota}$, based on the intervals $U_{\iota}=\left[a_{\iota},b_{\iota}\right]$ of length $\frac{11}{30}$, whose left and right endpoints are displayed as green and red stars, respectively. Each function $W_{4,\iota}$ is based on a set $C_{\iota}$ of 10 points $\xi^{\iota}_{k}$, which are displayed as yellow circles. In the left figure, the points $\xi^{\iota}_{\kappa}$ are equispaced in the interval $U_{\iota}$, while in the right figure the sets $C_\iota$ and $C_{\lambda}$ share the same points on the overlap $U_\iota \cap U_{\lambda}$.}
    \label{plot1606}
\end{figure}
% \begin{equation}\label{Ciota}
%     C_{\iota}=\left\{\xi^{\iota}_{\kappa}=a_{\iota}+\frac{b_{\iota}-a_{\iota}}{K+1}\kappa \, :\, \kappa=1,\dots,K \right\}, \quad a_{\iota}= \min U_{\iota}, \ b_{\iota}=\max U_{\iota}.
% \end{equation}
% As a result of the construction, we have 
% \begin{equation}\label{cst}
% b_{\iota}-a_{\iota}=r_d, \qquad \iota=1,\dots,M.  
% \end{equation}
% and 
% \begin{equation} \label{boundM}
%     M\leq \frac{2(b-a)}{r_d}.
% \end{equation}

Finally, we introduce the following rational approximation operator
\begin{equation}\label{finalop}
    \mathcal{Q}_{\mathcal{H},\mu}[f](x)=\sum\limits_{\iota=1}^M W_{\mu,\iota}(x)p_{\iota}(x), \qquad \mu\in 2\mathbb{Z}_{+},
\end{equation}
where $p_{\iota}(x)$ are the local histopolation polynomials of the function $f$ based on the sets $\mathcal{S}_{\iota}$ and $W_{\mu,\iota}(x)$ are the univariate \textit{multinode Shepard functions} based on the sets $C_{\iota}$. 
\begin{remark}\label{remarkimp}
    Since the set $\mathcal{S}_{\iota}$, $\iota=1,\dots,M$, defined in~\eqref{setSiota}, consists of chained intervals, the histopolation polynomial $p_{\iota}(x)$ is unique~\cite{Bruno:2023:PIO}. For this reason, the approximation operator~\eqref{finalop} is well-defined.
\end{remark}

\begin{remark}
    The approximation operator $\mathcal{Q}_{\mathcal{H},\mu}[f]$, defined in~\eqref{finalop}, is not a histopolation operator, that is, in general, it does not interpolate the integral data used in its construction, i.e.
     \begin{equation*}
  \mu_i\left(\mathcal{Q}_{\mathcal{H},\mu}[f]\right)\neq \mu_i(f), \qquad i=1,\dots,n,     
    \end{equation*}
where $\mu_i$ are the linear functionals defined in~\eqref{intefun}. 
% As expected, the integrals of $\mathcal{Q}_{\mathcal{H},\mu}[f]$ over the segments $s_i$ approximate the values $\mu_i(f)$ with the same order of accuracy as $\mathcal{Q}_{\mathcal{H},\mu}[f]$ uniformly approximates $f$ over those segments. 
\end{remark}

\begin{proposition}
The operator $\mathcal{Q}_{\mathcal{H},\mu}[f]$ is a quasi-histopolation operator.  
\end{proposition}
\begin{proof}
    To prove this, by the definition of quasi-histopolation~\cite{Wang:2004:QIW}, it is sufficient to show that 
\begin{equation*}
 \mathcal{Q}_{\mathcal{H},\mu}[p](x)=p(x), \qquad x\in [a,b],
 \end{equation*}
 for any polynomial $p$ of degree less than or equal to $\displaystyle{\min_{\iota=1,\dots,M} \operatorname{deg}(p_{\iota}(x))}\geq d.$
Now, for any $\iota=1,\dots,M$, by Remark~\ref{remarkimp}, the histopolation polynomial on the set $\mathcal{S}_{\iota}$ is unique and then we have
\begin{equation*}
    p_{\iota}(p,x)=p(x), \qquad x\in [a,b].
\end{equation*}
By~\eqref{finalop} and leveraging the partition of unity property~\eqref{partofuni} of the multinode Shepard functions, it results
\begin{equation*}
\mathcal{Q}_{\mathcal{H},\mu}[p](x)=\sum_{\iota=1}^M  W_{\mu,\iota}\left( x\right) p_{\iota}\left(p, x\right)= \sum_{\iota=1}^M  W_{\mu,\iota}\left(x\right) p(x)=p(x), \qquad x\in [a,b].  
\end{equation*}  
\end{proof}

\begin{remark}
    In the process of construction of quasi-histopolation, avoiding extrapolation requires accurately capturing the behavior of the local polynomial histopolant within the interval defined by the input data. To achieve this, the numerical support of the multinode Shepard function must be entirely contained within, and remain close to, this interval. Figure~\ref{functionnord} shows that $U_{\ell,j}$ must have the same length $r_d$ (or $h^{\max,d}_{\mathcal{S}_n}$) to ensure this feature. The Lemma~\ref{lemmaboundBmu} in the next section proves that this condition is also sufficient.
\end{remark}
\begin{figure}
    \centering
    \includegraphics[width=0.5\linewidth]{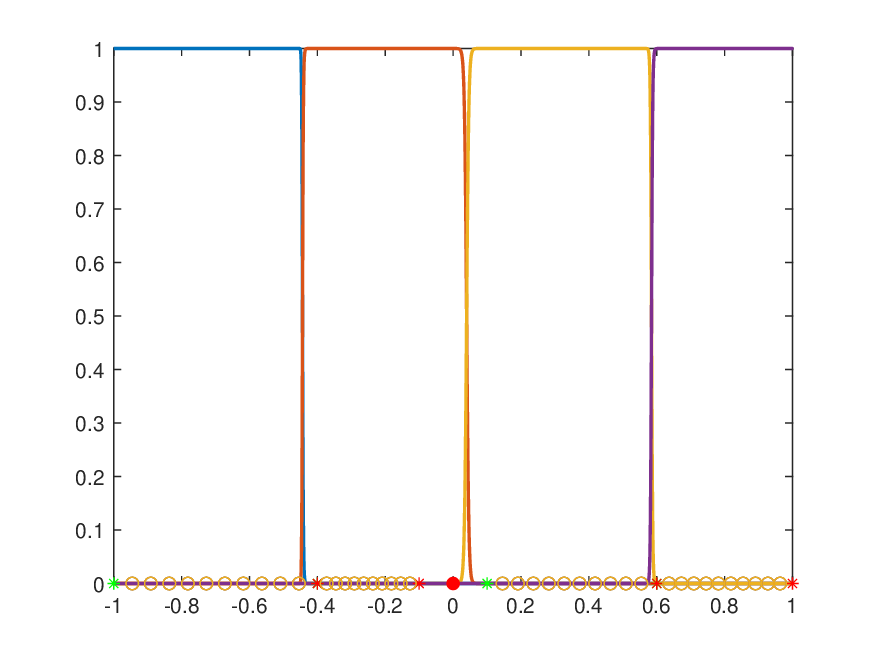}
    \caption{Plot of the multinode Shepard functions $W_{4,{\iota}}$ based on sets of $10$ equispaced points $\xi^{\iota}_{\kappa}$, $\iota=1,\dots,4$, in the open intervals $U_{1}=(-1,-0.4)$, $U_{2}=(-0.4,-0.1)$, $U_{3}=(0.1,0.6)$ and $U_{4}=(0.6,1)$, respectively. The point $0 \in [-0.1,0.1]$ is assumed to be a discontinuity point. The numerical support of $W_{4,2}$ includes this discontinuity while simultaneously capturing the behavior of the local polynomial $p_2$. As a consequence, it leads to a poor extrapolation.}
    \label{functionnord}
\end{figure}

\section{Error bound}\label{sec3}
In this section, we compute an error bound of
the quasi-histopolation operator~\eqref{finalop} on a proper subset
\begin{equation}\label{Xi}
\Omega= \bigcup\limits_{\ell=0}^{m+1} \left[\xi^{\ell}_1,\xi^{n_\ell}_K\right] \subsetneq \bigcup\limits_{\ell=0}^{red}{m+1} I_{\ell},    
\end{equation}
in the case of equispaced nodes $X_n=\left\{a+i\frac{b-a}{n}, \, i=0,\dots,n \right\} \subset\left[a,b\right]$. As in Section~\ref{sec2}, let $d\in \mathbb{N}_0$ denote the degree of the local histopolation polynomials and $r_d>0$ as defined in~\eqref{rd}. The minimal covering of each interval of continuity $I_{\ell}$ is then realized by intervals of the family
\begin{equation*}
    \mathcal{F}_{\ell} = \left\{ \left[x_i,x_i+\left(d+1\right)\dfrac{b-a}{n}\right] \right\}_{x_i \in I_{\ell}}
\end{equation*}
each of which intersects the subsequent one in only a node, with the exception of the rightmost interval, eventually, whose intersection with the previous one may consist of more than one point.
 For the sake of simplicity, we assume that the rightmost interval intersects the previous one in only one node.

The following Lemma, which refers to the simplest non-trivial case of two multinode Shepard functions based on the points $\left\{\xi^1_{k}\right\}$, $\left\{\xi^2_{k}\right\}$, $k=1,\dots,K$, of two consecutive intervals of equal length (see Fig.~\ref{xleft} and Fig.~\ref{xright}), is useful to provide bounds of the multinode Shepard function $W_{\mu,\iota}(x)$, $\iota=1,2$, for different relative positions of $x \in \Omega$.

\begin{figure}
    \centering
  \includegraphics[width=0.8\linewidth]{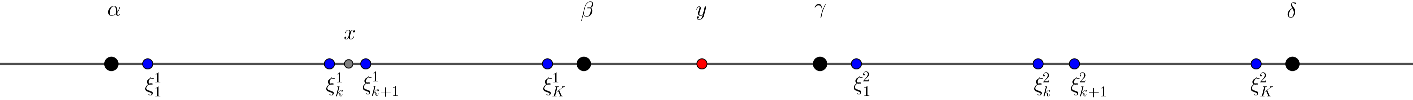}
    \caption{Plot of $K$ equispaced points (in blue) for each of the sets $C_1$ (in the left) and $C_2$ (in the right). The point $x\in \left(\xi^1_k,\xi^1_{k+1}\right)$ is displayed in gray, while the discontinuity point is displayed in red.}
    \label{xleft}
\end{figure}
\begin{figure}
    \centering
  \includegraphics[width=0.8\linewidth]{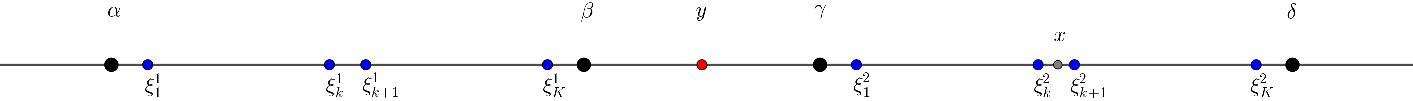}
    \caption{Plot of $K$ equispaced points (in blue) for each of the sets $C_1$ (in the left) and $C_2$ (in the right). The point $x\in (\xi^2_k,\xi^2_{k+1})$ is displayed in gray while the discontinuity point is displayed in red.}
    \label{xright}
\end{figure}
\begin{lemma}\label{lemmaboundBmu}
Let $\alpha,\beta,\gamma,\delta\in \mathbb{R}$ be such that $\alpha<\beta\leq\gamma<\delta$ with
\begin{equation}\label{abc}
  \beta-\alpha=\delta-\gamma=r_d
\end{equation}
 and let  
\begin{equation*}
  C_{1}=\left\{\xi^{1}_{\kappa}=\alpha+\frac{\beta-\alpha}{K+1}\kappa \, :\, \kappa=1,\dots,K \right\} \subset (\alpha,\beta), 
\end{equation*}  
\begin{equation*}
  C_{2}=\left\{\xi^{2}_{\kappa}=\gamma+\frac{\delta-\gamma}{K+1}\kappa \, :\, \kappa=1,\dots,K \right\} \subset (\gamma,\delta).
\end{equation*}
Then, for any $x\in \left(\xi^1_{k},\xi^1_{k+1}\right)$, $k=1,\dots,K-1$, we have 
\begin{equation} \label{Bmu2boundgen}
    W_{\mu,2}(x)\leq \left(\frac{k!(K-k)!}{\prod\limits_{\kappa=1}^K \left(K-k+\kappa\right)+(K+1)^K\left(\frac{\gamma-\beta}{r_d}\right)^K}\right)^{\mu}=F_{\mu}(\gamma-\beta,r_d,k,K).
\end{equation}
Similarly, for any $x\in \left(\xi^2_{k},\xi^2_{k+1}\right)$, $k=1,\dots,K-1$, we have 
\begin{equation} \label{Bmu1boundgen}
    W_{\mu,1}(x) \leq \left(\frac{k!(K-k)!}{\prod\limits_{\kappa=1}^K \left(k+\kappa\right)+(K+1)^K\left(\frac{\gamma-\beta}{r_d}\right)^K}\right)^{\mu}=F_{\mu}(\gamma-\beta,r_d,K-k,K). 
\end{equation} 
\end{lemma}
\begin{proof}
 Let $x\in \left(\xi^1_{k},\xi^1_{k+1}\right)$, $k=1,\dots,K-1$. Then we have
 \begin{eqnarray*}
     W_{\mu,2}(x)&=&\frac{\prod\limits_{\kappa=1}^{K} \left\lvert x-\xi^2_{{\kappa}} \right\rvert^{-\mu}}{\prod\limits_{\kappa=1}^{K} \left\lvert x-\xi^1_{{\kappa}} \right\rvert^{-\mu}+\prod\limits_{\kappa=1}^{K} \left\lvert x-\xi^2_{{\kappa}} \right\rvert^{-\mu}} \leq \frac{\prod\limits_{\kappa=1}^{K} \left\lvert x-\xi^1_{{\kappa}} \right\rvert^{\mu}}{\prod\limits_{\kappa=1}^{K} \left\lvert x-\xi^2_{{\kappa}} \right\rvert^{\mu}} \\
     &\leq& \frac{\left\lvert x-\xi^1_{1} \right\rvert^{\mu} \cdots \left\lvert x-\xi^1_{k} \right\rvert^{\mu} \cdot \left\lvert x-\xi^1_{k+1} \right\rvert^{\mu} \cdots \left\lvert x-\xi^1_{K} \right\rvert^{\mu}}{\left\lvert \xi^1_{k+1}-\xi^2_{1} \right\rvert^{\mu} \cdot \left\lvert \xi^1_{k+1}-\xi^2_{2} \right\rvert^{\mu} \cdots \left\lvert \xi^1_{k+1}-\xi^2_{K-1} \right\rvert^{\mu} \cdot \left\lvert \xi^1_{k+1}-\xi^2_{K} \right\rvert^{\mu}} \\
      &\leq&  \left(\frac{(\beta-\alpha)^{K}\frac{1+k-1}{K+1}\cdots \frac{1+0}{K+1}\cdot\frac{1+0}{K+1}\cdots \frac{1+K-k-1}{K+1}}{\left(\frac{(\beta-\alpha)(K-k)+(\delta-\gamma)}{K+1}+(\gamma-\beta)\right)\cdots\left(\frac{(\beta-\alpha)(K-k)+(\delta-\gamma)K}{K+1}+(\gamma-\beta)\right)}\right)^{\mu} \\ &\leq&  \left(\frac{(\beta-\alpha)^{K}k!(K-k)!}{\prod\limits_{\kappa=1}^K \left(\left(\beta-\alpha\right)\left(K-k\right)+\kappa\left(\delta-\gamma\right)\right)+(K+1)^K(\gamma-\beta)^K}\right)^{\mu} \\
      &\leq&  \left(\frac{k!(K-k)!}{\prod\limits_{\kappa=1}^K \left(K-k+\kappa\right)+(K+1)^K\left(\frac{\gamma-\beta}{r_d}\right)^K}\right)^{\mu}
 \end{eqnarray*}
where in the last inequality we used the hypothesis~\eqref{abc}.
Similarly, if $x\in \left(\xi^2_{k},\xi^2_{k+1}\right)$, $k=1,\dots,K-1$, we have 
\begin{eqnarray*}
W_{\mu,1}(x)&=&\frac{\prod\limits_{\kappa=1}^{K} \left\lvert x-\xi^1_{{\kappa}} \right\rvert^{-\mu}}{\prod\limits_{\kappa=1}^{K} \left\lvert x-\xi^1_{{\kappa}} \right\rvert^{-\mu}+\prod\limits_{\kappa=1}^{K} \left\lvert x-\xi^2_{{\kappa}} \right\rvert^{-\mu}} \leq \frac{\prod\limits_{\kappa=1}^{K} \left\lvert x-\xi^2_{{\kappa}} \right\rvert^{\mu}}{\prod\limits_{\kappa=1}^{K} \left\lvert x-\xi^1_{{\kappa}} \right\rvert^{\mu}} \\
&\leq& \frac{\left\lvert x-\xi^2_{1} \right\rvert^{\mu} \cdots \left\lvert x-\xi^2_{k} \right\rvert^{\mu} \cdot \left\lvert x-\xi^2_{k+1} \right\rvert^{\mu} \cdots \left\lvert x-\xi^2_{K} \right\rvert^{\mu}}{\left\lvert \xi^2_{k}-\xi^1_{1} \right\rvert^{\mu} \cdot \left\lvert \xi^2_{k}-\xi^1_{2} \right\rvert^{\mu} \cdots \left\lvert \xi^2_{k}-\xi^1_{K-1} \right\rvert^{\mu} \cdot \left\lvert \xi^2_{k}-\xi^1_{K} \right\rvert^{\mu}} \\
&\leq&  \left(\frac{(d-c)^{K}\frac{1+k-1}{K+1}\cdots \frac{1+0}{K+1}\cdot\frac{1+0}{K+1}\cdots \frac{1+K-k-1}{K+1}}{\left(\frac{(\delta-\gamma)k+(\beta-\alpha)K}{K+1}+(\gamma-\beta)\right)\cdots\left(\frac{(\delta-\gamma)k+(\beta-\alpha)}{K+1}+(\gamma-\beta)\right)}\right)^{\mu} \\
&\leq&  \left(\frac{(\delta-\gamma)^{K}k!(K-k)!}{\prod\limits_{\kappa=1}^K \left(\left(\delta-\gamma\right)k+\kappa\left(\beta-\alpha\right)\right)+(K+1)^K(\gamma-\beta)^K}\right)^{\mu}\\
&\leq&  \left(\frac{k!(K-k)!}{\prod\limits_{\kappa=1}^K \left(k+\kappa\right)+(K+1)^K\left(\frac{\gamma-\beta}{r_d}\right)^K}\right)^{\mu}.    
\end{eqnarray*}
\end{proof}

\begin{remark}
    The strength of  method hinges on the ability of the multinode Shepard functions to closely approximate the characteristic functions of the equally sized intervals \(U_{\iota}\) that support the local histopolants (see Fig.~\ref{intersectionofWi}). Lemma~\ref{lemmaboundBmu} guarantees this property under the given assumptions.
\end{remark}
\begin{figure}
    \centering
    \includegraphics[width=8cm]{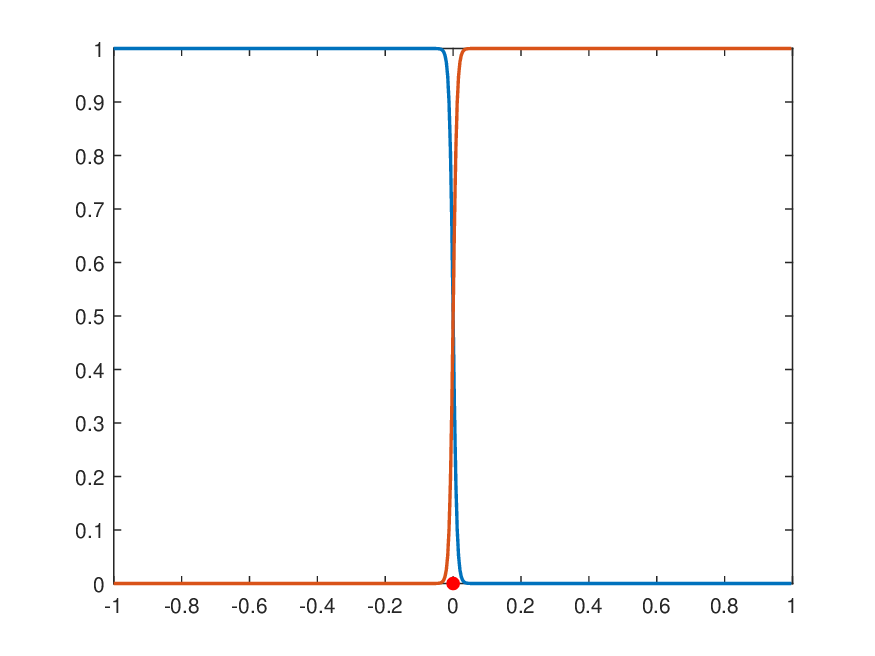}
    \caption{Plot of the multinode Shepard functions $W_{4,1}$, $W_{4,2}$ based on sets of $10$ points $\xi^{\iota}_{\kappa}$, $\iota=1,2$,  equispaced in the interval $(-1,-0.1)$ and $(0.1,1)$, respectively. The point $0 \in [-0.1,0.1]$ is assumed to be a discontinuity point. The numerical support of $W_{4,1}$ and $W_{4,2}$ is contained within, and remains close to, the interval $U_1=[-1,-0.1]$ and $U_{2}=[0.1,1]$.}
    \label{intersectionofWi}
\end{figure}

\begin{proposition} \label{proposition}
    For any $K\in\mathbb{N}$ and $k=1,\dots,K-1$, we have 
    \begin{equation*}
         F_{\mu}(\gamma-\beta,r_d,k,K)\leq F_{\mu}(\gamma-\beta,r_d,K-1,K).
    \end{equation*}
\end{proposition}
\begin{proof}
  We notice that, for any $K \in \mathbb{N}$, we have       
    \begin{equation*}
        k!(K-k)! \leq (K-1)!  \qquad k=1,\dots,K-1.
    \end{equation*}
    In fact
    \begin{eqnarray*}
        \frac{k!(K-k)!}{(K-1)!} = \dfrac{\prod\limits_{j=0}^{k-2}(k-j)}{\prod\limits_{j=0}^{k-2}(K-1-j)} = \prod\limits_{j=0}^{k-2}\dfrac{k-j}{K-1-j} \leq 1.
    \end{eqnarray*} 
    Therefore
    \begin{eqnarray*}
        F_{\mu}(\gamma-\beta,r_d,k,K)&=& \left(\frac{k!(K-k)!}{\prod\limits_{\kappa=1}^K \left(K-k+\kappa\right)+(K+1)^K\left(\frac{\gamma-\beta}{r_d}\right)^K}\right)^{\mu} \\
        &\leq& \left(\frac{(K-1)!}{\prod\limits_{\kappa=1}^K \left(1+\kappa\right)+(K+1)^K\left(\frac{\gamma-\beta}{r_d}\right)^K}\right)^{\mu}\\
        &=& F_{\mu}(\gamma-\beta,r_d,K-1,K).
    \end{eqnarray*}
\end{proof}
\begin{lemma}\label{reimanderhistopolationlemma}
    Let $\alpha,\beta \in \mathbb{R}$ such that $\alpha<\beta$. 
    Let $ f \in C^{\nu-1} \left([\alpha,\beta]\right)$ and suppose that $ f^{(\nu)} (x) $ exists at each point $ x \in (\alpha,\beta)$. Let $p^{\Sigma}_{\nu-1}\in\mathbb{P}_{\nu-1}$ be the histopolant of the function $f$ on the unisolvent set of segments $ \Sigma:= \{ \sigma_1, \ldots, \sigma_{\nu} \} $, that is 
    \begin{equation}\label{hyplemma}
       \int_{\sigma_i} p^{\Sigma}_{\nu-1}(x) d x=\int_{\sigma_i} f(x) d x, \qquad i=1,\dots,\nu.
    \end{equation} 
   Then, there exist $ \zeta_{1}, \ldots, \zeta_{\nu} \in [\alpha,\beta] $ and $\eta \in (\alpha,\beta)$ such that
    \begin{equation} \label{eq:remainderhistpolation}
        f(x) - p^{\Sigma}_{\nu-1}(x) = \frac{\prod\limits_{i=1}^{\nu} (x-\zeta_i)}{\nu!} f^{(\nu)} (\eta), \qquad x \in [\alpha,\beta].
    \end{equation}   
\end{lemma}

\begin{proof}
    By using the hypothesis~\eqref{hyplemma} and the linearity of the integral, we get   
\begin{equation*}
\int_{\sigma_i} \left(f(x)-p^{\Sigma}_{\nu-1}(x)\right) d x = 0, \qquad i=1,\dots,\nu.  
\end{equation*}
By the mean value theorem, for each $ i = 1, \ldots, \nu $, there exists $ \zeta_i \in \sigma_i $ such that 
\begin{equation*}
f(\zeta_i)=p^{\Sigma}_{\nu-1}(\zeta_i), \qquad i=1,\dots,\nu.    
\end{equation*}
The proof is now identical to that of the classical polynomial interpolation, see for instance \cite[Th. $ 3.1.1 $]{Davis:1975:IAA}.
\end{proof}

Finally, we are ready to present an error bound for the operator~\eqref{finalop}.
\begin{theorem}\label{theoremerrorboundp}
    Let $k_{\max}=\max\limits_{\iota=1,\dots,M} k_{\iota}$ and $ f \in C^{k_{\max}}(\Omega)$. Then
\begin{equation}\label{errboundp}
     \left\lvert f(x)-\mathcal{Q}_{\mathcal{H},\mu}[f](x) \right\rvert  \leq  \max_{\substack{\iota=1,\dots,M \\ x\in U_{\iota}}}     \frac{\prod\limits_{i=1}^{k_{\iota}} \left \lvert x-\zeta^{\iota}_i\right\rvert}{k_{\iota}!} 
    \max\limits_{\eta \in \displaystyle{\cap_{\substack{\iota=1,\dots,M \\ x\in U_{\iota}}} U_{\iota}}}
    \left \lvert f^{(k_{\iota})}  (\eta) \right \rvert
\end{equation}
\begin{equation*}
    +  M
    \left
    (\left\lvert f({x})\right\rvert + \max_{\substack{\iota=1,\dots,M \\ x\notin U_{\iota}}}\left\lvert p_{\iota}(x)\right\rvert \right)  \max_{\substack{\iota=1,\dots,M \\ x\notin U_{\iota}}}F_{\mu}\left(a_{\iota+1}-b_{\iota},r_d, K-1,K\right).
\end{equation*}
     The second part of~\eqref{errboundp} can be made arbitrarily small by increasing $K$.
\end{theorem}
\begin{proof}
We consider $x\in \Omega$. Then $ x \in \left[\xi^{\Bar{\ell}}_1,\xi^{n_{\Bar{\ell}}}_K\right]$,
for some $\Bar{\ell}\in \{0,\dots,m+1\}$. From~\eqref{finalop} and by employing both the properties~\eqref{partofuni} and~\eqref{nonnegprop}, it follows 

	\begin{eqnarray}	
	e({x})&=&\left\lvert f({x})- \mathcal{Q}_{\mathcal{H},\mu}[f](x)\right\rvert \notag \nonumber \\
	&=&\left\lvert \sum_{\iota=1}^M	W_{\mu ,\iota}\left( {x}\right)f({x})-\sum_{\iota=1}^M	W_{\mu,\iota}\left(x\right)p_{\iota}(x)\right\rvert \notag  \nonumber \\
	&\leq&\sum_{\iota=1}^M\left\lvert f({x})-p_{\iota}(x)\right\rvert W_{\mu ,\iota}\left(x\right)  \label{sumdifffpol} \\
    &\leq &  \sum_{\substack{\iota=1,\dots,M \\ x\in U_{\iota}}}\left\lvert f({x})-p_{\iota}(x)\right\rvert W_{\mu ,\iota}\left(x\right) +\sum_{\substack{\iota=1,\dots,M \\ x\notin U_{\iota}}} \left
    (\left\lvert f({x})\right\rvert + \left\lvert p_{\iota}(x)\right\rvert \right)
    W_{\mu ,\iota}\left(x\right) \nonumber \\
    &\le&   \max_{\substack{\iota=1,\dots,M \\ x\in U_{\iota}}}     \frac{\prod\limits_{i=1}^{k_{\iota}} \left \lvert x-\zeta^{\iota}_i\right\rvert}{k_{\iota}!} 
    \max\limits_{\eta \in \displaystyle{\cap_{\substack{\iota=1,\dots,M \\ x\in U_{\iota}}} U_{\iota}}}
    \left \lvert f^{(k_{\iota})}  (\eta) \right \rvert \nonumber\\ 
    &+&  \left
    (\left\lvert f({x})\right\rvert + \max_{\substack{\iota=1,\dots,M \\ x\notin U_{\iota}}}\left\lvert p_{\iota}(x)\right\rvert \right) \sum_{\substack{\iota=1,\dots,M \\ x\notin U_{\iota}}}  W_{\mu ,\iota}(x) \nonumber
    \\
    &\le&   \max_{\substack{\iota=1,\dots,M \\ x\in U_{\iota}}}     \frac{\prod\limits_{i=1}^{k_{\iota}} \left \lvert x-\zeta^{\iota}_i\right\rvert}{k_{\iota}!} 
    \max\limits_{\eta \in \displaystyle{\cap_{\substack{\iota=1,\dots,M \\ x\in U_{\iota}}} U_{\iota}}}
    \left \lvert f^{(k_{\iota})}  (\eta) \right \rvert \nonumber
    \end{eqnarray}
    \begin{eqnarray}
        \phantom{e(x)} &+& \left
    (\left\lvert f({x})\right\rvert + \max_{\substack{\iota=1,\dots,M \\ x\notin U_{\iota}}}\left\lvert p_{\iota}(x)\right\rvert \right) \sum_{\substack{\iota=1,\dots,M \\ x\notin U_{\iota}}}  F_{\mu}\left(a_{\iota+1}-b_{\iota},r_d, K-1,K\right) \nonumber \\
    &\leq& \max_{\substack{\iota=1,\dots,M \\ x\in U_{\iota}}}     \frac{\prod\limits_{i=1}^{k_{\iota}} \left \lvert x-\zeta^{\iota}_i\right\rvert}{k_{\iota}!} 
    \max\limits_{\eta \in \displaystyle{\cap_{\substack{\iota=1,\dots,M \\ x\in U_{\iota}}} U_{\iota}}}
    \left \lvert f^{(k_{\iota})}  (\eta) \right \rvert \nonumber \\ &+&  M\left
    (\left\lvert f({x})\right\rvert + \max_{\substack{\iota=1,\dots,M \\ x\notin U_{\iota}}}\left\lvert p_{\iota}(x)\right\rvert \right)  \max_{\substack{\iota=1,\dots,M \\ x\notin U_{\iota}}}F_{\mu}\left(a_{\iota+1}-b_{\iota},r_d, K-1,K\right), \nonumber
\end{eqnarray}
where in the last two inequalities we have used Lemma~\ref{lemmaboundBmu}. 
\end{proof}

As a consequence of the previous Theorem, the following result presents an error bound of the operator~\eqref{finalop} in the infinity norm on $\Omega$
\begin{equation*}
    \left \lVert f \right\rVert_{\infty,\Omega} = \sup\limits_{x \in \Omega} \left\lvert f(x) \right\rvert.
\end{equation*}
For this purpose, let us denote by:
\begin{itemize}
\item[-]  $\mathcal{R}_{r_d}(y)=\left\{{x}\in \mathbb{R}\, : \, y-r_d<x\le y+r_d\right\},$ for a fixed $y\in\left[a,b\right]$, the half-open interval of radius $r_d$ centered at $y$;
  \item[-] $ \mathcal{M}_{r_d}= \sup\limits_{y\in [a,b]}\mathrm{card}\left(\left\{C_{\iota}\, :\,  C_{\iota} \cap \mathcal{R}_{r_d}({{y}}) \neq \emptyset, \, \iota=1,\dots,M\right\}\right)$,
the supremum \\ number of the sets $C_{\iota}$ intersecting $\mathcal{R}_{r_d}(y)$;
\item[-] $\mathcal{M}_{f}=\max\limits_{\iota=1,\dots,M} \max\limits_{\ell=0,\dots,m+1} \max\limits_{y\in I_{\ell}}\left\lvert f^{\left(k_{\iota}\right)} (y) \right\rvert$,
\item[-] $h^{\min}_{\mathcal{S}^{\prime}_m}=\min\limits_{j=0,\dots,m} \left\{ \left\lvert s_{\sigma\left(j\right)}\right\rvert \right\}$.
\end{itemize}

\begin{remark} \label{remarkMrd}
We observe that $\mathcal{M}_{r_d}\leq 2$ if $y=\frac{a_{\iota}+b_{\iota}}{2}$. In all remaining cases, $\mathcal{M}_{r_d}\leq 4$.
\end{remark}

\begin{theorem} \label{errorbound}
   Let $ f \in C^{k_{\max}}(\Omega) $ and assume that $\mu>\frac{k_{\max}+1}{K}$. Then
    \begin{equation*}
        \left\lVert f-\mathcal{Q}_{\mathcal{H},\mu}[f] \right\rVert_{\infty,\Omega}  \leq \mathcal{E}_1+\mathcal{E}_2, 
    \end{equation*}
    where
        \begin{equation*}
        \mathcal{E}_1=\mathcal{M}_{f} \mathcal{M}_{r_d} \max\limits_{\iota=1,\dots,M}\frac{ (2r_d)^{k_{\iota}}}{k_{\iota}!} \left( 1+2\cdot2^{k_{\max}}+\sum\limits_{\theta=2}^{\infty} 2 (\theta+1)^{k_{\max}} 	\frac{1}{(\theta-\frac{1}{2})^{K\mu}}  \right) 
    \end{equation*}    
   tends to zero when $h^{\max}_{\mathcal{S}_n}\rightarrow0, \, n\rightarrow \infty $, equivalently $r_d \rightarrow 0$, and 
    \begin{equation*}
     \mathcal{E}_2 =  2M \left\lVert f \right\rVert_{\infty} \left(1 + \max_{\ell=0,\dots,m}\max_{\substack{\iota=1,\dots,M \\ U_{\iota}\cap I_{\ell} = \emptyset}} \max\limits_{y\in[a,b]}\lambda_{k_{\iota}}(y) \right)  F_{\mu}\left(h^{\min}_{\mathcal{S}^{\prime}_m},r_d,K-1,K\right)
    \end{equation*}
    can be made arbitrarily small by increasing $K$. 
\end{theorem}
\begin{proof}
We consider $x\in \Omega$. Then $ x \in \left[\xi^{\Bar{\ell}}_1,\xi^{n_{\Bar{\ell}}}_K\right]$,
for some $\Bar{\ell}\in \{0,\dots,m+1\}$. The following subsets $\{A_{\theta}\}_{\theta=0}^{\Theta}$  
      represent the minimal set that covers the interval $\left[\xi^{\Bar{\ell}}_1,\xi^{n_{\Bar{\ell}}}_K\right]$ 
      \begin{equation*}
    A_0=\mathcal{R}_{r_d}(x):=\left\{{y}\in \mathbb{R}\, : \, x-r_d<y\le x+r_d
         \right\}
\end{equation*}
and
      \begin{equation*}
A_{\theta}=\mathcal{R}_{r_d}({x}- 2 r_d\theta)\cup \mathcal{R}_{r_d}({x}+ 2 r_d \theta), \qquad \theta=1,\dots, \Theta.
\end{equation*} 
Obviously 
\begin{equation*}
    \left[\xi^{\Bar{\ell}}_1,\xi^{n_{\Bar{\ell}}}_K\right] \subset \bigcup_{\theta=0}^{\Theta}A_{\theta}.
\end{equation*}
As $A_{\theta}$ consists of two identical copies of $\mathcal{R}_{r_d}(x)$ for $\theta = 1, 2, \dots, \Theta$, the number of subsets $C_{\Bar{\ell},j}$, $j = 0, \dots, n_{\Bar{\ell}}$, with at least one vertex in $A_{\theta}$, $\theta = 1, 2, \dots, \Theta$, is bounded by $2 \mathcal{M}_{r_d}$.
By~\eqref{sumdifffpol}, it results 
\begin{eqnarray}\label{err1}
    e({x}) &\leq& \sum_{\iota=1}^M\left\lvert f({x})-p_{\iota}(x)\right\rvert W_{\mu ,\iota}\left(x\right) \notag \\ &=&  \sum_{\substack{\iota=1,\dots,M \\ U_{\iota}\cap I_{\Bar{\ell}} \neq \emptyset}}\left\lvert f({x})-p_{\iota}(x)\right\rvert W_{\mu ,\iota}\left(x\right) + 
      \sum_{\substack{\iota=1,\dots,M \\ U_{\iota}\cap I_{\Bar{\ell}} = \emptyset}}\left\lvert f({x})-p_{\iota}(x)\right\rvert W_{\mu ,\iota}\left(x\right) \\ \label{e1,2}
      &=:& e_1(x)+e_2(x).
\end{eqnarray}
To begin, we focus on bounding $e_1$.
Let $V_0$ denote the set of $C_{\iota}$ containing at least one node in $A_0$. Similarly, for $\theta = 1, \dots, \Theta$, let $V_{\theta}$ denote the set of $C_{\iota}$ that contain at least one node in $A_{\theta}$ but no nodes in $A_{\theta-1}$. As established, we get  
\begin{equation*}
    \bigcup_{\theta=0}^\Theta V_{\theta}=\bigcup_{\iota=1}^M C_{\iota} \quad \text{ and } \quad \bigcap_{\theta=0}^\Theta V_{\theta}=\emptyset.
\end{equation*} 
Under the assumption
\begin{equation*}
x\in\bigcup\limits_{\substack{\iota=1,\dots,M \\ U_{\iota} \cap I_{\Bar{\ell}} \neq \emptyset}} C_{\iota},
\end{equation*} 
 by the Kronecker delta property of the multinode Shepard functions~\cite{DellAccio:2018:ROF}, from equation~\eqref{err1} and since $f \in C^{k_{\max}}(\Omega)$, leveraging Lemma~\ref{reimanderhistopolationlemma}, we get
\begin{equation}\label{err1B1}
    e_1(x)\leq \left\lvert f({x})-p_{\iota} (f,x)\right\rvert\leq   \frac{\prod\limits_{i=1}^{k_{\iota}} \left\lvert x-\zeta^{\iota}_i\right\rvert}{k_{\iota}!} f^{(k_{\iota})} (\eta_{\iota}),\qquad x \in \mathbb{R},
\end{equation}
where $\min \left\{\zeta^{\iota}_1, \dots, \zeta^{\iota}_{k_{\iota}},x  \right\} \leq \eta_{\iota} \leq \max \left\{\zeta^{\iota}_1, \dots, \zeta^{\iota}_{k_{\iota}},x  \right\}$.
In the case of
\begin{equation*}
    {x}\notin \bigcup\limits_{\substack{\iota=1,\dots,M \\ U_{\iota} \cap I_{\bar{\ell}} \neq \emptyset}} C_{\iota},
\end{equation*}
 we consider $\iota_{\min}\in \{1,\dots, M\}$ so that 
\begin{equation*}
\prod\limits_{\kappa=1}^{K}\left\lvert {x}-{\xi}^{\iota_{\min}}_{\kappa}\right\rvert= \underset{\iota=1,\dots,M}{\min} 	\prod\limits_{\kappa=1}^{K}\left\vert {x}-{\xi}^{\iota}_{\kappa}\right\vert.   
\end{equation*}
We notice that
\begin{equation}\label{ineqinR0}
\prod\limits_{\kappa=1}^{K}\left\lvert{x}-{\xi}^{\iota}_{\kappa}\right\rvert\leq  (2r_d)^{K},    
\end{equation}
if $C_{\iota} \in V_0$ whereas 
\begin{equation}\label{ineqinRthe}
((2\theta-1) r_d)^{K} \leq \prod\limits_{\kappa=1}^{K}\left\vert {x}-{\xi}^{\iota}_{\kappa}\right\vert
 \leq \left((2\theta+2) r_d \right)^{K},    
\end{equation}
if $C_{\iota} \in V_{\theta}$, $\theta=1,\dots, \Theta$.
Finally, we get
\begin{eqnarray*}
		W_{\mu ,\iota}(x)&=&\dfrac{\prod\limits_{\kappa=1}^{K}\left\vert 
{x}-{\xi}^{\iota}_{\kappa}\right\vert ^{-\mu }}{\sum\limits_{\tau=1}^{M}%
\prod\limits_{\kappa=1}^{K}\left\vert {x}-{\xi}^{\tau}_{\kappa}\right\vert
^{-\mu}}\leq \min \left\{  1,\prod\limits_{\kappa=1}^{K}\frac{\left\lvert{x}-{\xi}^{\iota}_{\kappa}\right\rvert^{-\mu}}{\left\lvert {x}-{\xi}^{\iota_{\min}}_{\kappa}\right\rvert^{-\mu}}\right\} \\
&\leq& \left\{ \begin{array}{ll}
			1, & C_{\iota} \in V_{\theta}, \, \theta=0,1,\\
			\frac{1}{(\theta-\frac{1}{2})^{K\mu}},& C_{\iota} \in V_{\theta}, \theta \geq 2.
		\end{array}\right. 
\end{eqnarray*} 
We highlight that, given $C_{\iota} \subset U_{\iota}$, the inequalities~\eqref{ineqinR0} and~\eqref{ineqinRthe} remain valid for any selected set $N_{\iota} \subset U_{\iota}$, where $\iota = 1, \dots, M$.
From~\eqref{err1},~\eqref{err1B1}, it directly results
\begin{eqnarray*}
    	e_1(x) & \leq& \mathcal{M}_{f} \left( \sum\limits_{C_{\iota} \in V_{0}\cup V_{1}}   \frac{\prod\limits_{i=1}^{k_{\iota}} \left\lvert x-\zeta^{\iota}_i \right \rvert}{k_{\iota}!} +\sum\limits_{\theta=2}^{\Theta} \sum\limits_{C_{\iota} \in V_{\theta}}\frac{\prod\limits_{i=1}^{k_{\iota}} \left\lvert x-\zeta^{\iota}_i \right \rvert}{k_{\iota}!} 	\frac{1}{(\theta-\frac{1}{2})^{K\mu}}  \right) \\
     & \leq& \mathcal{M}_{f} \left( \sum\limits_{C_{\iota} \in V_{0}}\frac{ (2r_d)^{k_{\iota}}}{k_{\iota}!}+\sum\limits_{C_{\iota} \in V_{1}}\frac{ (4r_d)^{k_{\iota}}}{k_{\iota}!}
     +\sum\limits_{\theta=2}^{\Theta} \sum\limits_{C_{\iota} \in V_{\theta}}\frac{\left((2\theta+2) r_d \right)^{k_{\iota}}}{k_{\iota}!} 	\frac{1}{(\theta-\frac{1}{2})^{K\mu}} \right) \\
      & \leq& \mathcal{M}_{f} \mathcal{M}_{r_d} \max\limits_{\iota=1,\dots,M}\frac{ (2r_d)^{k_{\iota}}}{k_{\iota}!} \left( 1+2\cdot2^{k_{\max}}+\sum\limits_{\theta=2}^{\Theta} 2 (\theta+1)^{k_{\max}} 	\frac{1}{(\theta-\frac{1}{2})^{K\mu}}  \right) 
     \end{eqnarray*}
For $\mu > \frac{k_{\max}+1}{K}$ we have the convergence of $\sum\limits_{\theta=2}^{\infty} \frac{(\theta+1)^{k_{\max}}}{(\theta-\frac{1}{2})^{K\mu}}$.
To bound $e_2(x)$, defined in~\eqref{e1,2}, note that since $x \in \left[\xi^{\Bar{\ell}}_1,\xi^{n_{\Bar{\ell}}}_K\right]$, there exist $i = \Bar{\ell}, \dots, n_{\Bar{\ell}}$ and $k = 1, \dots, K-1$ such that $x \in \left[\xi^i_k, \xi^i_{k+1}\right]$.
To estimate the multinode Shepard functions $W_{\mu,\iota}$, we apply Lemma~\ref{lemmaboundBmu} and Proposition~\ref{proposition} twice.
 In more detail, for $\iota = 1, \dots, M$ such that $U_{\iota} \cap I_{\Bar{\ell}} = \emptyset$, we use the following substitutions 
\begin{equation*}
  \begin{cases}
\alpha \to a_{\iota-1}, \quad \beta \to b_{\iota-1}, \quad \gamma \to a_{\iota}, \quad \delta \to b_{\iota}, 
& \text{ for } \xi^{\iota}_K < \xi^i_k, \\
\alpha \to a_{\iota}, \quad \beta \to b_{\iota}, \quad \gamma \to a_{\iota+1}, \quad \delta \to b_{\iota+1},
& \text{ for }\xi^{\iota}_1 > \xi^i_{k+1}.
\end{cases}
\end{equation*}
 Consequently, we have 
\begin{eqnarray*}
    e_2(x) &=& \sum_{\substack{\iota=1,\dots,M \\ U_{\iota}\cap I_{\Bar{\ell}} = \emptyset}}\left\lvert f({x})-p_{\iota}(x)\right\rvert W_{\mu ,\iota}\left(x\right) \\ 
    &\leq& \sum_{\substack{\iota=1,\dots,M \\ U_{\iota}\cap I_{\Bar{\ell}} = \emptyset}} \left
    (\left\lvert f({x})\right\rvert + \left\lvert p_{\iota}(x)\right\rvert \right)
    W_{\mu ,\iota}\left(x\right) \\
    &\leq& \left
    (\left\lvert f({x})\right\rvert + \max_{\substack{\iota=1,\dots,M \\ U_{\iota}\cap I_{\Bar{\ell}} = \emptyset}}\left\lvert p_{\iota}(x)\right\rvert \right) \sum_{\substack{\iota=1,\dots,M \\ U_{\iota}\cap I_{\Bar{\ell}} = \emptyset}}  W_{\mu ,\iota}(x) \\
    &\leq& \left
    (\left\lVert f \right\rVert_{\infty} + \max_{\substack{\iota=1,\dots,M \\ U_{\iota}\cap I_{\Bar{\ell}} = \emptyset}}\left\lVert p_{\iota}(\cdot)\right\rVert_{\infty} \right) \times\\
    && \left( c_1 \cdot F_{\mu}\left(a_{\iota}-b_{\iota-1},r_d,K-k,K\right)  + c_2 \cdot F_{\mu}\left(a_{\iota+1}-b_{\iota},r_d,k,K\right)\right) \\
    &\leq& 
    M \left\lVert f \right\rVert_{\infty} \left(1 + \max_{\substack{\iota=1,\dots,M \\ U_{\iota}\cap I_{\Bar{\ell}} = \emptyset}} \max\limits_{y\in[a,b]}\lambda_{k_{\iota}}(y) \right) \times 
    \\
    &&
\left(F_{\mu}\left(h^{\min}_{\mathcal{S}^{\prime}_m},r_d,1,K\right)+F_{\mu}\left(h^{\min}_{\mathcal{S}^{\prime}_m},r_d,K-1,K\right)\right)\\
    &\leq& 
    2M \left\lVert f \right\rVert_{\infty} \left(1 + \max_{\ell=0,\dots,m}\max_{\substack{\iota=1,\dots,M \\ U_{\iota}\cap I_{\ell} = \emptyset}} \max\limits_{y\in[a,b]}\lambda_{k_{\iota}}(y) \right)  F_{\mu}\left(h^{\min}_{\mathcal{S}^{\prime}_m},r_d,K-1,K\right), 
\end{eqnarray*}
where 
$$c_1=\mathrm{card}\left(\left\{\iota=1,\dots,M : U_{\iota}\cap I_{\Bar{\ell}} =\emptyset \wedge \xi^{\iota}_K < \xi^i_k  \right\}\right),$$

$$c_2=\mathrm{card}\left(\left\{\iota=1,\dots,M : U_{\iota}\cap I_{\Bar{\ell}} =\emptyset \wedge \xi^{\iota}_1 > \xi^i_{k+1}  \right\}\right)$$
 and $\max\limits_{y\in[a,b]}\lambda_{k_{\iota}}(y)$ is the Lebesgue constant relative to the Lebesgue function $\lambda_{k_{\iota}}(y)$ for Lagrange interpolation in $N_{\iota}$.
\end{proof}

\begin{corollary}
    Under the same assumptions and settings of the Theorem~\ref{errorbound}, we get 
    \begin{equation*}
        \left\lvert \int_{s_i}f(x)dx - \int_{s_i} \mathcal{Q}_{\mathcal{H},\mu}[f](x)dx \right\rvert \leq \left(\mathcal{E}_1+\mathcal{E}_2 \right) h^{\max}_{\mathcal{S}_n}.
    \end{equation*}
\end{corollary}
\begin{proof}
By easy computations, we get
\begin{equation*}
\begin{aligned}
    \left\lvert \int_{s_i}f(x)dx - \int_{s_i} \mathcal{Q}_{\mathcal{H},\mu}[f](x)dx \right\rvert &= \left\lvert \int_{s_i}\left(f(x)dx - \mathcal{Q}_{\mathcal{H},\mu}[f](x)\right)dx  \right\rvert \\
    \leq  \int_{s_i} \left\lvert f(x)dx - \mathcal{Q}_{\mathcal{H},\mu}[f](x) \right\rvert dx &\leq \left(\mathcal{E}_1+\mathcal{E}_2 \right) h^{\max}_{\mathcal{S}_n}.
\end{aligned}
  \end{equation*}
\end{proof}
 \section{Numerical results}\label{sec4}
In the following, we present several numerical tests that prove the effectiveness of the introduced quasi-histopolation method. In all experiments, $X_{n}$ is considered the set of $n+1$ equispaced nodes in $[-1,1]$ and we use the Chebyshev polynomial of the first kind given by
\begin{equation*}
    T_0(x)=1, \quad T_1(x)=x, \quad T_{k+1}(x)=2x T_k(x)-T_{k-1}(x), \quad k\ge 2,
\end{equation*}
as basis polynomials for the respective spaces. In particular, each histopolation polynomial $p_{\iota}(x)$ is based on the dataset $\mathcal{S}_{\iota}$, for $\iota = 1, \dots, M$, and is expressed with respect to the Chebyshev polynomial basis of the first kind, scaled on the interval defined by the minimum left endpoint and the maximum right endpoint of all segments in $\mathcal{S}_{\iota}$.
 
\subsection{Numerical test $1$}
In the first numerical test, we use the following smooth functions in $\left[-1,1\right]$    
\begin{equation*}
    f_1(x)=\frac{1}{1+25x^2},
    \quad
    f_2(x)=\cos\left(5\pi x\right),
\quad
    f_3(x)=\frac{1}{x-1.5},
    \quad
    f_4(x)=\cos\left(50\pi x\right),
\end{equation*}
where $f_1$ is the famous Runge function scaled in the interval $[-1,1]$, $f_2$ is an oscillatory function with a relatively low frequency, $f_3$ is a rational function which has a singularity at $x=1.5$ and $f_4$ is an oscillatory function with a higher frequency compared to $f_2$. 
All these functions are often used for comparing numerical behaviors or evaluating approximation and interpolation methods, see, f.e.,~\cite{Bruno:2024:PHM}. 
In Figs.~\ref{Runge}-\ref{f3} are displayed the plots of the functions $f_i$, $i=1,2,3$, and the trend of the approximation error in the $L_1$-norm
\begin{equation*}
    e_1\left[f\right] = \int_{-1}^{1} \left\lvert f(x)-\mathcal{Q}_{\mathcal{H},\mu}[f](x)\right\rvert dx,
\end{equation*}
in the maximum norm 
\begin{equation}\label{erroremassimo}
    e_{\max}\left[f\right] = \max_{x \in X_{n_e}} \left\lvert f(x)-\mathcal{Q}_{\mathcal{H},\mu}[f](x)\right\rvert,
\end{equation}
and the trend of the mean approximation error
\begin{equation*}
    e_{\mathrm{mean}}\left[f\right] = \frac{1}{n_e} \sum\limits_{j=1}^{n_e}  \left\lvert f(x_j)-\mathcal{Q}_{\mathcal{H},\mu}[f](x_j)\right\rvert,
\end{equation*}
computed on the set of segments $\mathcal{S}_n$ based on different sets of equispaced nodes $X_n$, $n=100:100:1000$ by varying $d=3:3:9$. Specifically, we compute $e_{\max}$ and $e_{\mathrm{mean}}$ relative to a set of $n_e=10007$ equispaced nodes. In Fig.~\ref{cos50x}, for the function $f_4$ we use instead $n=200:200:2000$ and $d=3:3:15$.

\begin{figure}
    \centering
    \includegraphics[width=0.49\linewidth]{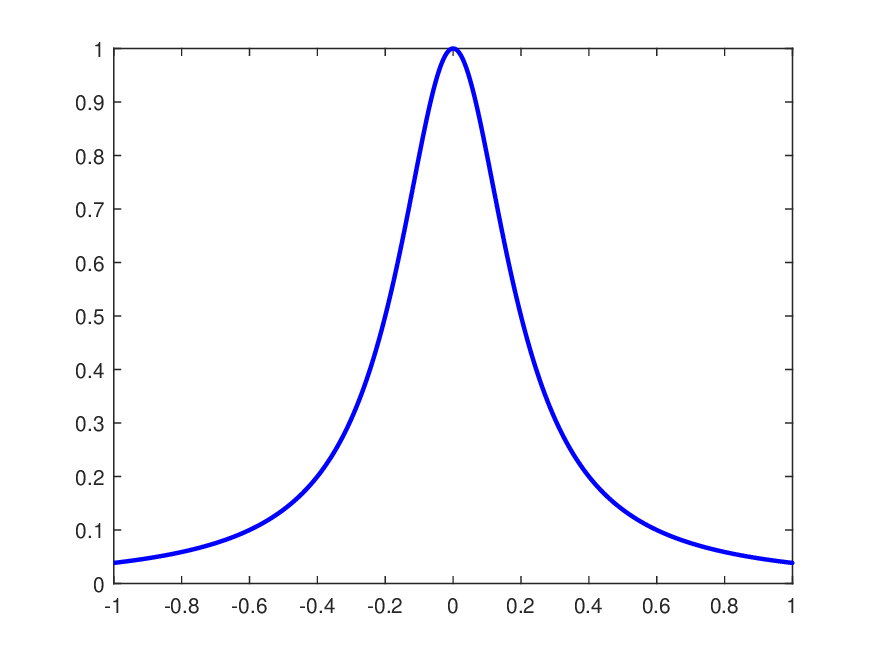}
    \includegraphics[width=0.49\linewidth]{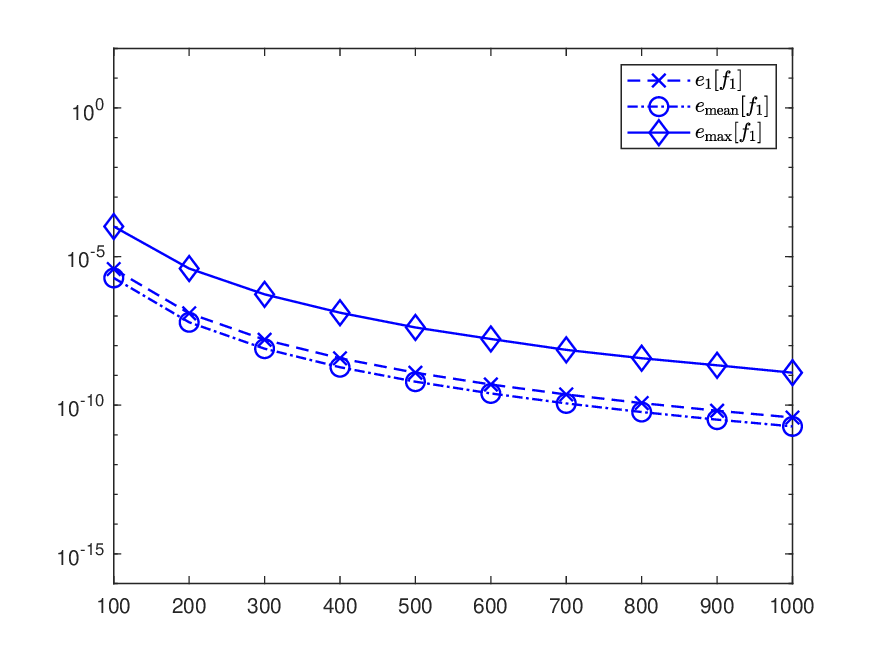}

    \includegraphics[width=0.49\linewidth]{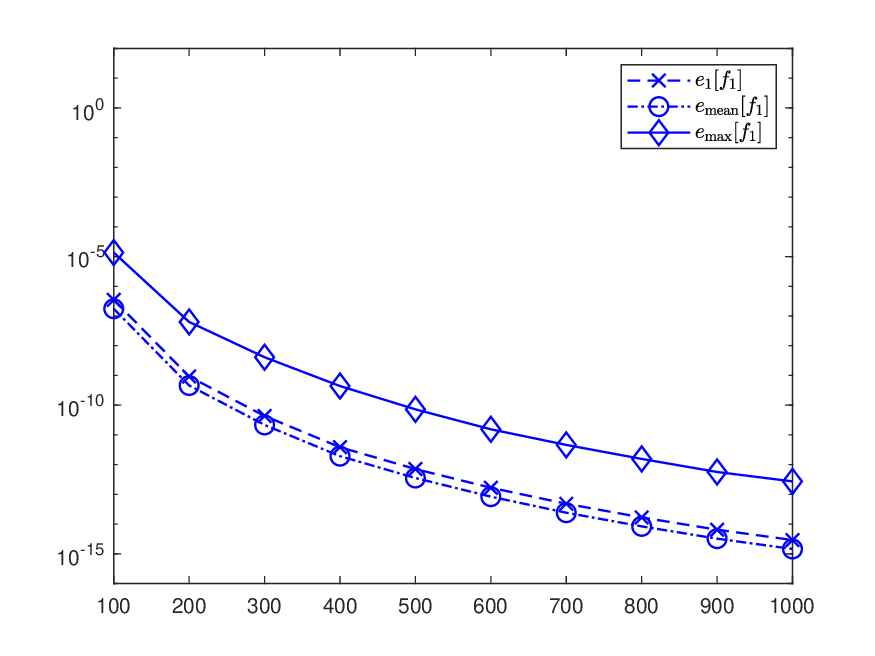}
    \includegraphics[width=0.49\linewidth]{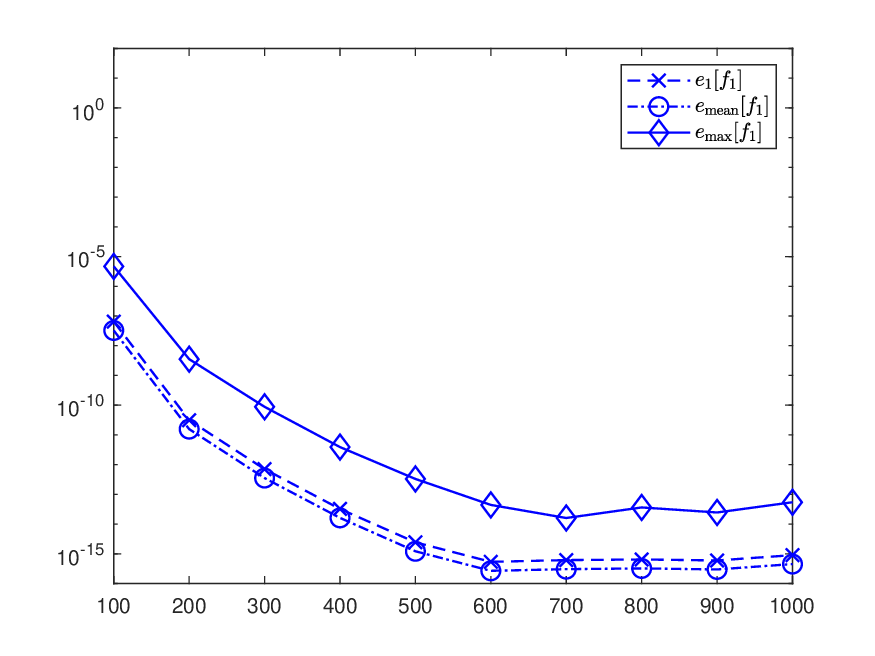}
    \caption{Plot of the function $f_1$ in $[-1,1]$ (up left) and trend of the maximum error, mean error and the error in $L_1$-norm computed on the set of segments $\mathcal{S}_n$ based on different sets of equispaced nodes $X_n$, $n=100:100:1000$ by fixing $d=3$ (up right), $d=6$ (bottom left), $d=9$ (bottom right).}
    \label{Runge}
\end{figure}

\begin{figure}
    \centering
    \includegraphics[width=0.49\linewidth]{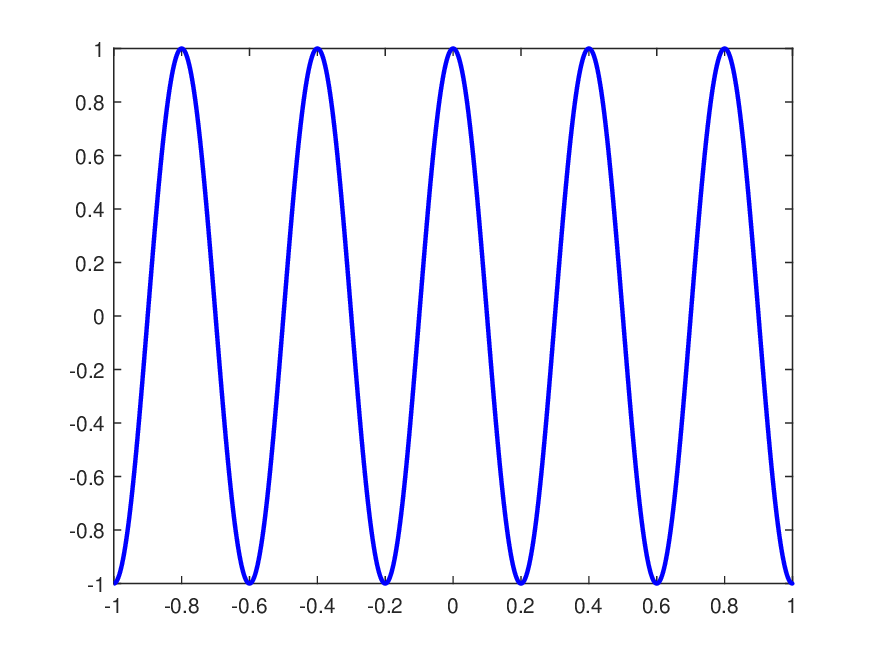}
    \includegraphics[width=0.49\linewidth]{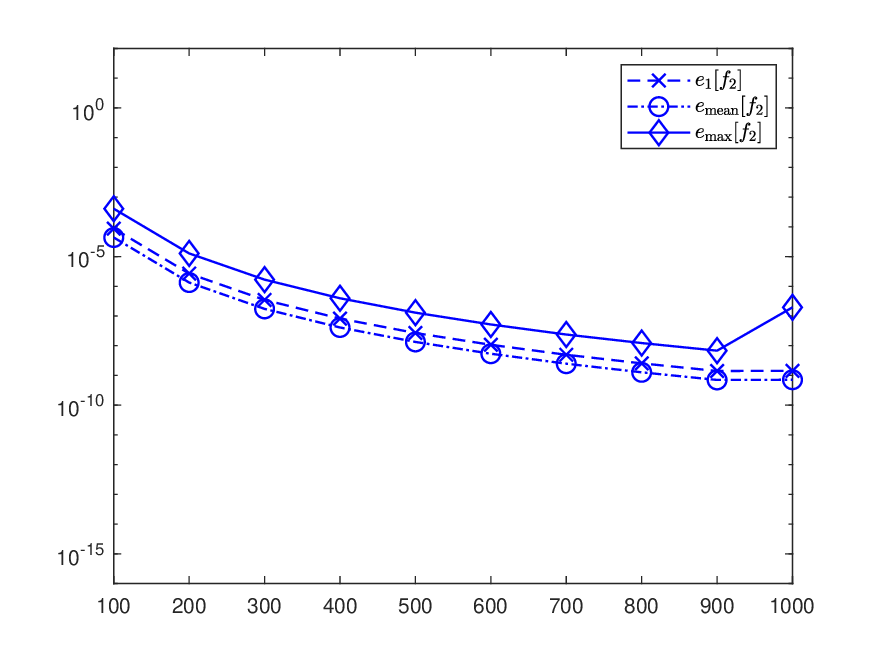}

    \includegraphics[width=0.49\linewidth]{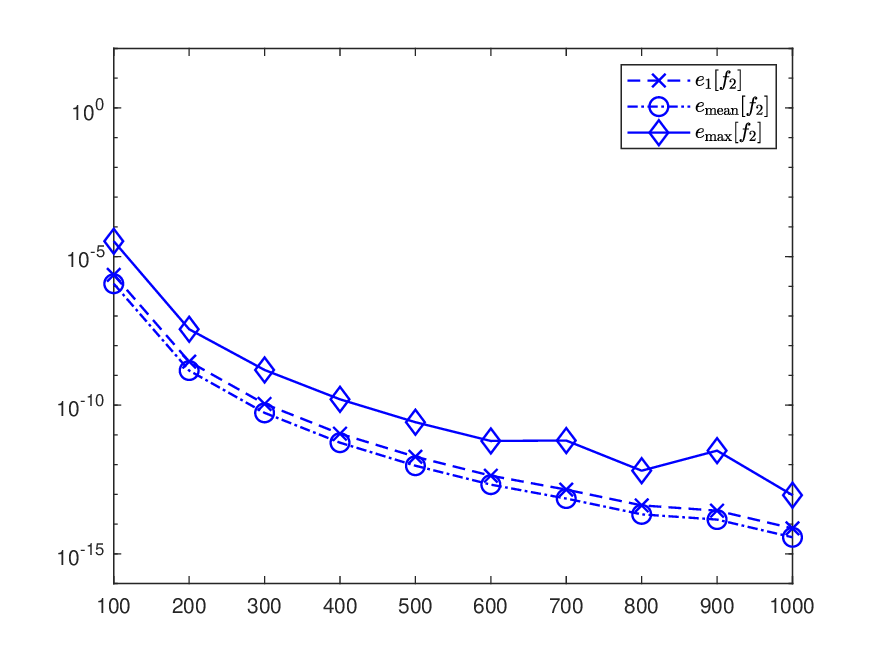}
    \includegraphics[width=0.49\linewidth]{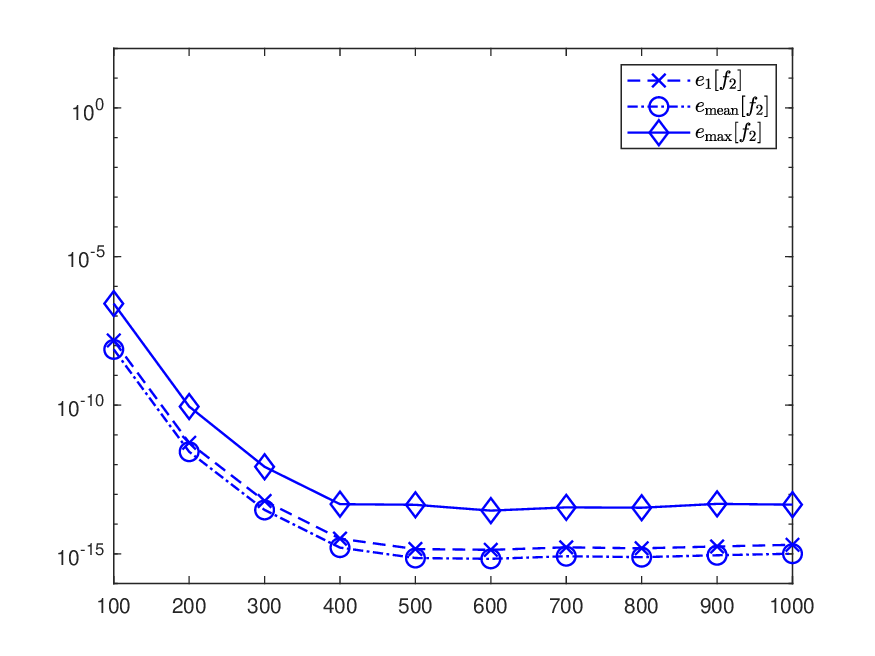}
    \caption{Plot of the function $f_2$ in $[-1,1]$ (up left) and trend of the maximum error, mean error and the error in $L_1$-norm computed on the set of segments $\mathcal{S}_n$ based on different sets of equispaced nodes $X_n$, $n=100:100:1000$ by fixing $d=3$ (up right), $d=6$ (bottom left), $d=9$ (bottom right).}
    \label{Cos5x}
\end{figure}

\begin{figure}
    \centering
    \includegraphics[width=0.49\linewidth]{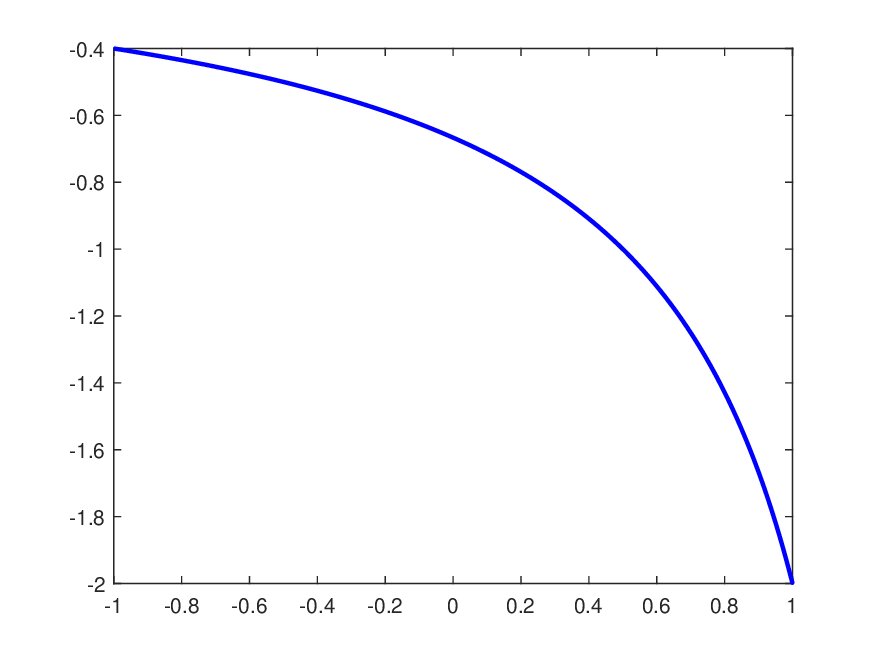}
    \includegraphics[width=0.49\linewidth]{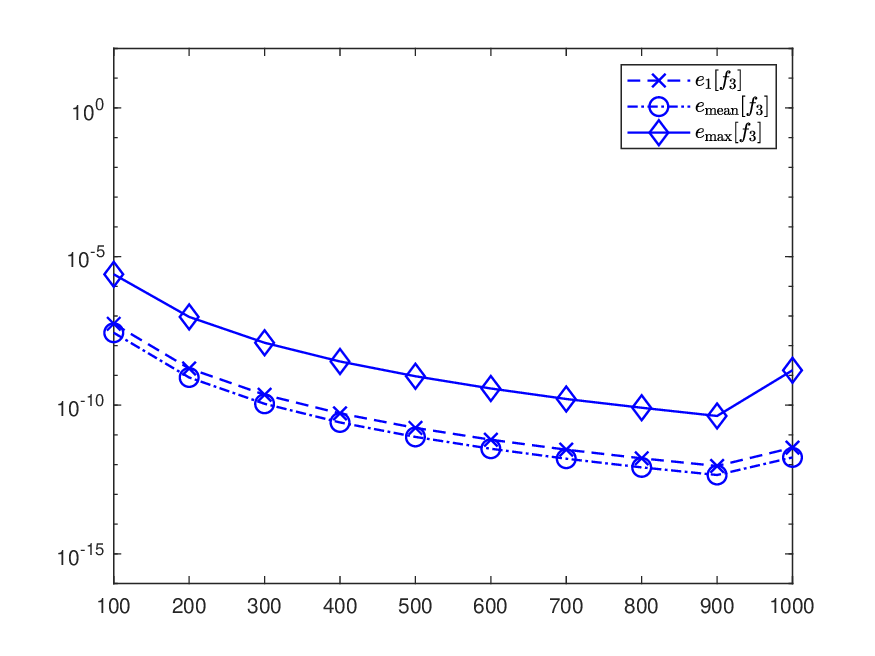}

    \includegraphics[width=0.49\linewidth]{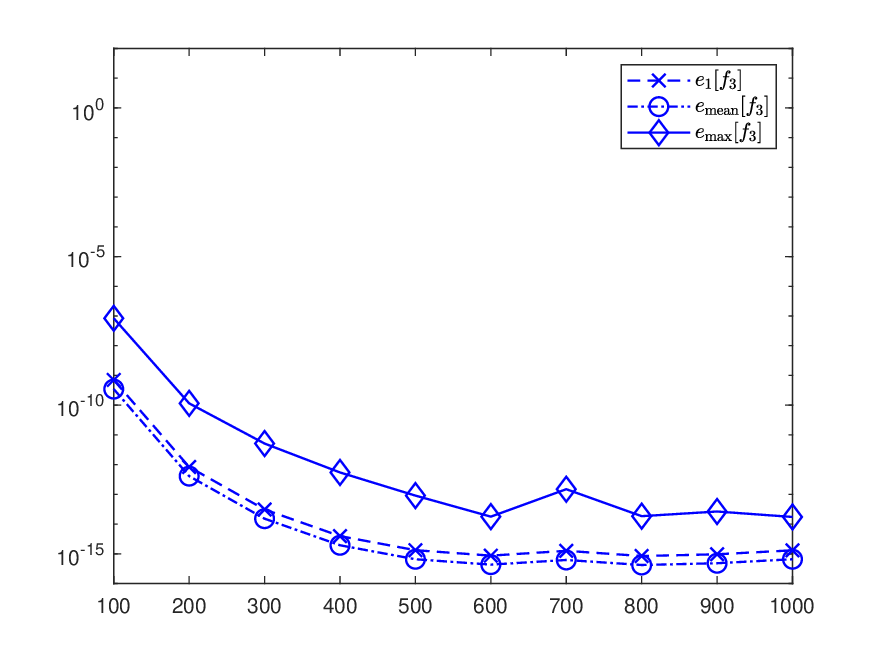}
    \includegraphics[width=0.49\linewidth]{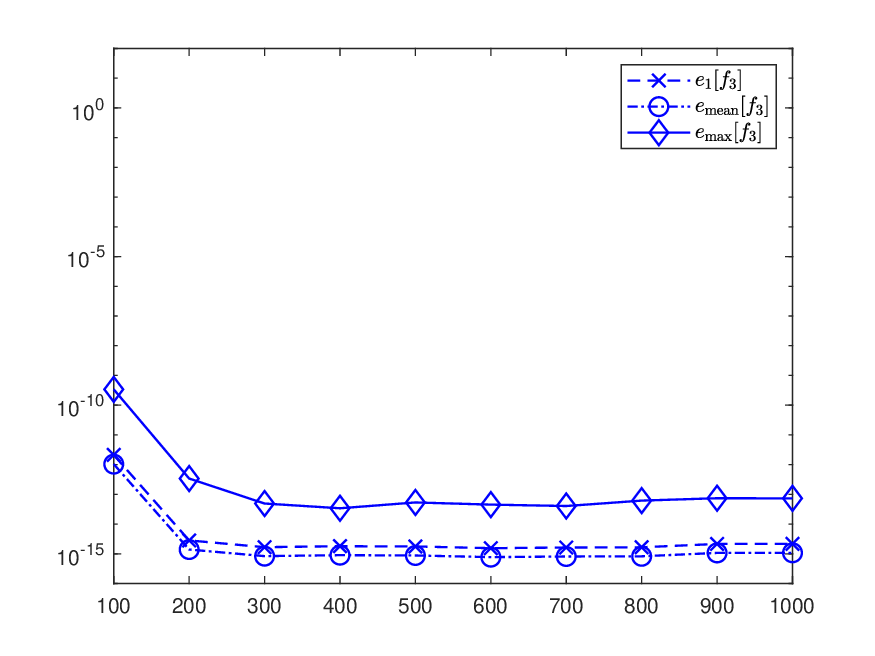}
    \caption{Plot of the function $f_3$ in $[-1,1]$ (up left) and trend of the maximum error, mean error and the error in $L_1$-norm computed on the set of segments $\mathcal{S}_n$ based on different sets of equispaced nodes $X_n$, $n=100:100:1000$ by fixing $d=3$ (up right), $d=6$ (bottom left), $d=9$ (bottom right).}
    \label{f3}
\end{figure}

\begin{figure}
    \centering
    \includegraphics[width=0.49\linewidth]{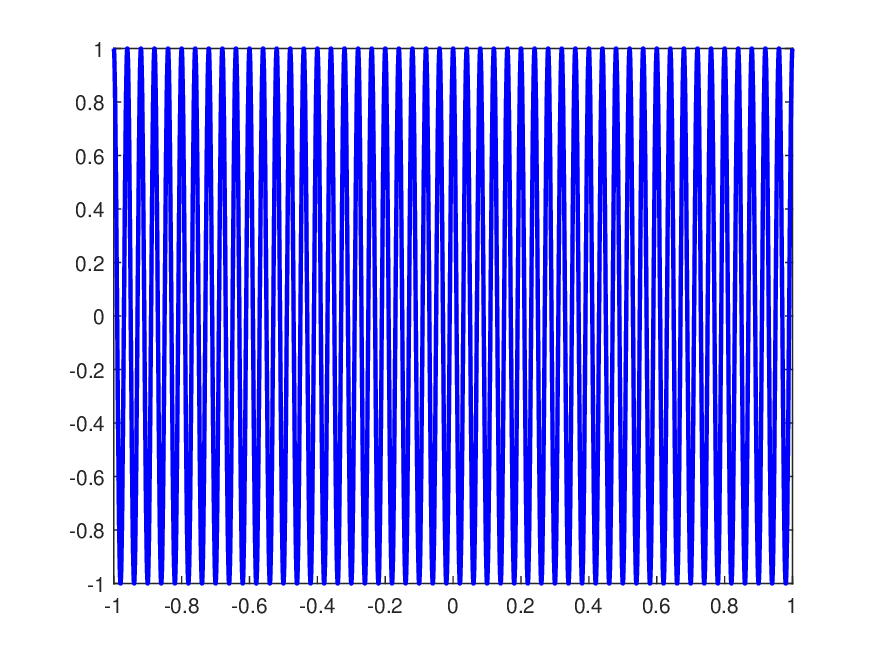}
    \includegraphics[width=0.49\linewidth]{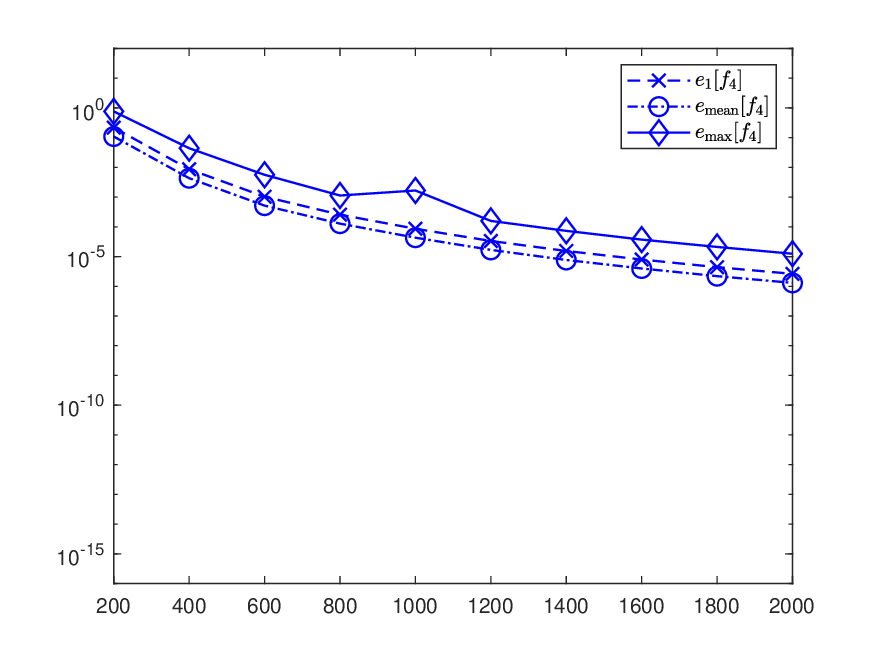}

    \includegraphics[width=0.49\linewidth]{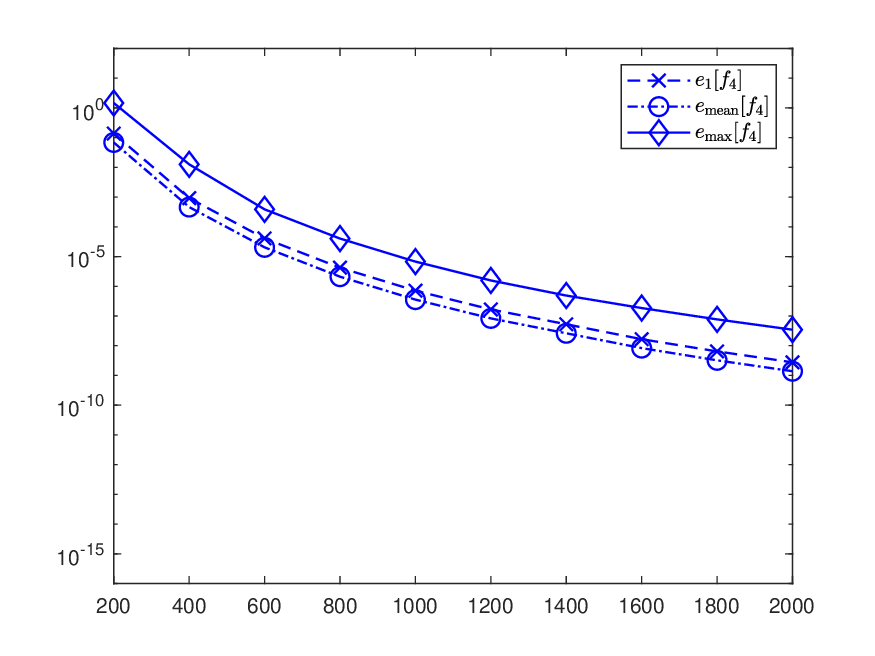}
    \includegraphics[width=0.49\linewidth]{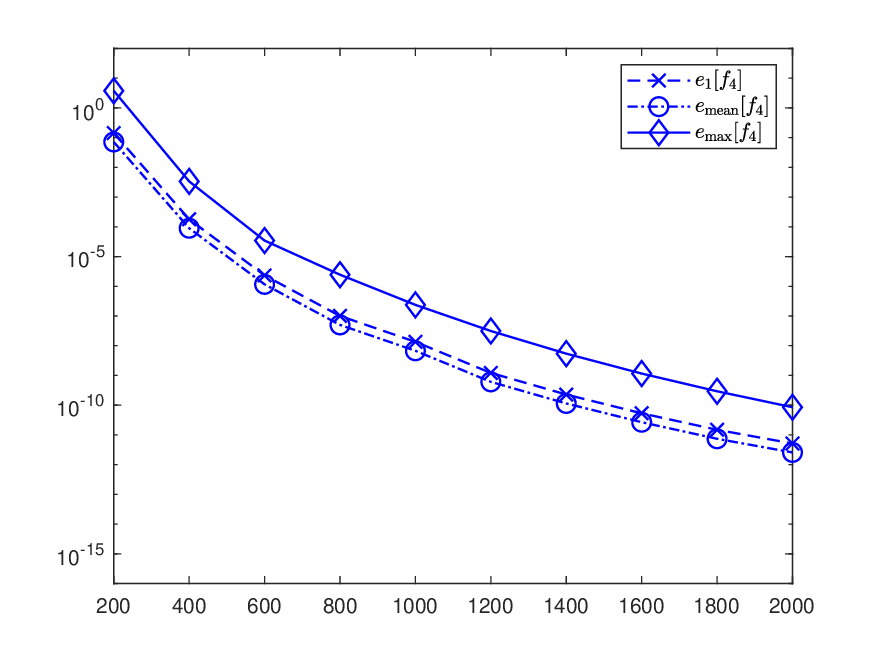}

    \includegraphics[width=0.49\linewidth]{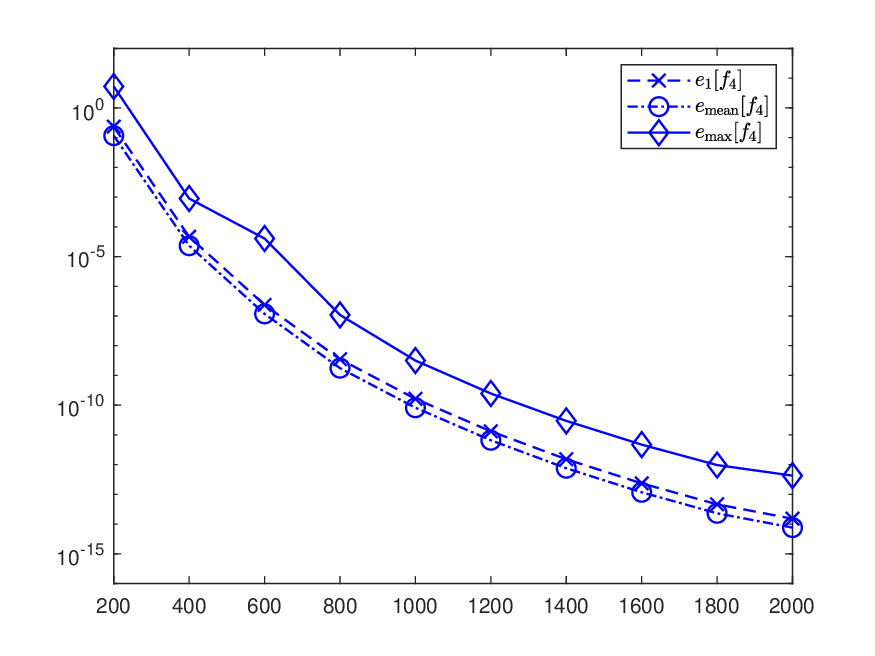}
    \includegraphics[width=0.49\linewidth]{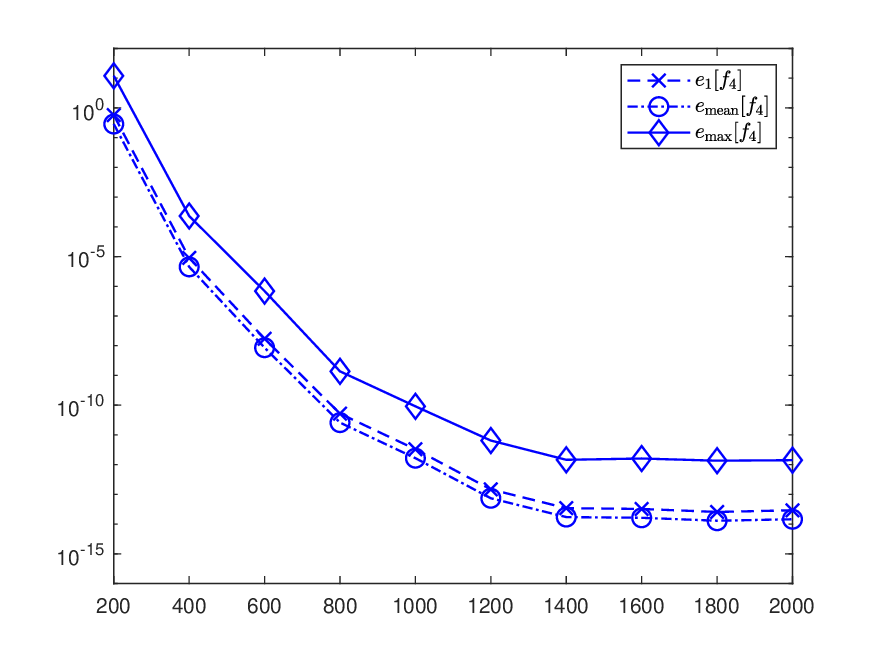}
    
    \caption{Plot of the function $f_4$ in $[-1,1]$ and trend of the maximum error, mean error and the error in $L_1$-norm computed on the set of segments $\mathcal{S}_n$ based on different sets of equispaced nodes $X_n$, $n=200:200:2000$ by varying $d=3:3:15$.}
    \label{cos50x}
\end{figure}
From these plots, it can be observed that the error in the $L_1$-norm lies between the mean approximation error and the maximum approximation error for any value of $n$ and any degree $d$. Moreover, the approximation errors consistently decrease as the number of equispaced nodes increases. Additionally, they diminish with an increase in the degree of the local histopolation polynomial until reaching the epsilon machine.

\subsection{Numerical test $2$}
In the second numerical test, we consider the following functions in $[-1,1]$
\begin{equation*}
    f_5(x)=\begin{cases}
        \sin\left(\frac{17}{8} \pi x\right) & x \leq 0, \\
        \frac{1}{2}\sin\left(\frac{17}{8} \pi x\right) + 10 & x>0,     
    \end{cases}
    \qquad
    f_6(x)=\begin{cases}
        \frac{1}{2} x^5-x^2 & x \leq 0, \\
        x^6-x^4+x^2-2 & x>0,     
    \end{cases} 
    \end{equation*}
    \begin{equation*}
    f_7(x)=\begin{cases}
        e^{\frac{1}{2}(x+1)} &  x \leq 0, \\
        1+e^{\frac{1}{4}(x+1)^2} & x>0,  
    \end{cases}
    \qquad
    f_8(x)=\begin{cases}
        \frac{5}{\left(\frac{x}{4}\right)^2+1} & \lvert x \rvert \geq \frac{1}{2}, \\
        \frac{3}{2} & -\frac{1}{2} < x < 0, \\
        \frac{1}{4} & 0 \leq x < \frac{1}{2},
    \end{cases}
\end{equation*}
where $f_5$, $f_6$ are defined in~\cite{Amat:2022:ACO}, $f_7$ in~\cite{Arandiga:2024:EAW} and $f_8$ in~\cite{Costarelli:2017:ADS}. 
To assess the quality of the approximation provided by the quasi-histopolation operator $\mathcal{Q}_{\mathcal{H},\mu}[f]$ defined in~\eqref{finalop}, in Figs.~\ref{plotf5}-\ref{plotf8}, we reconstruct the functions $f_i$, $i=5,6,7,8$, highlighting the behavior of the approximant near the singularity.
\begin{figure}
    \centering
    \includegraphics[width=0.49\linewidth]{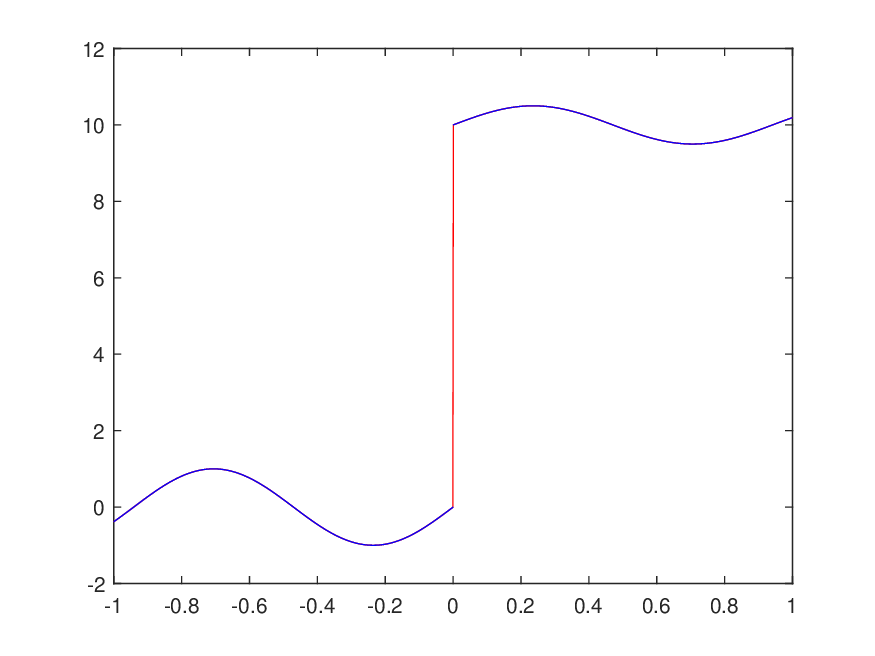}
    \includegraphics[width=0.49\linewidth]{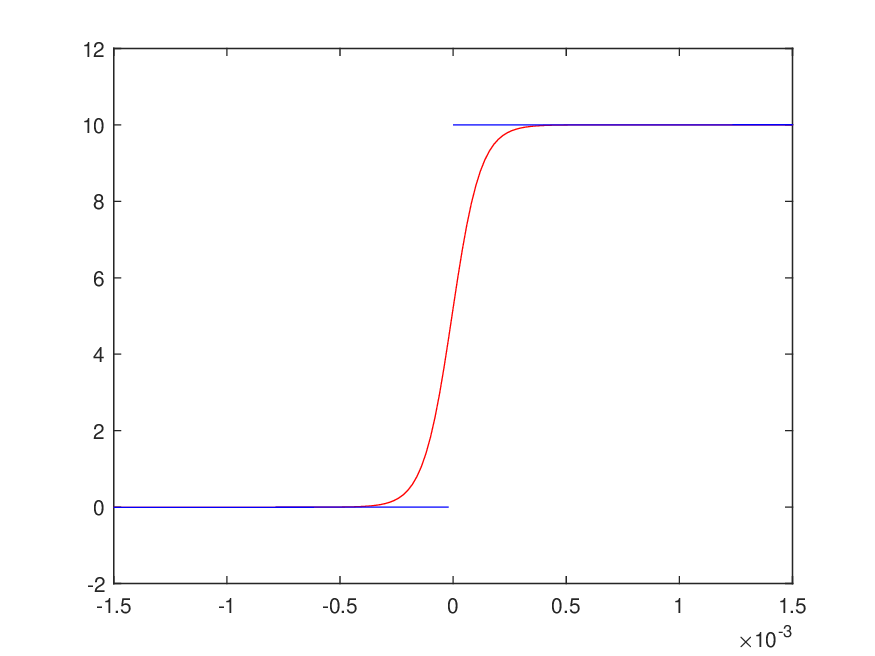}
    \caption{On the left, the plots of the function $f_5$ (blue) and of the quasi-histopolant $\mathcal{Q}_{\mathcal{H},\mu}[f_5]$ (red) computed on the set of segments $\mathcal{S}_n$ based on the set of equispaced nodes $X_n$, with $n=1025$ and by fixing $d=3$. On the right, a zoom near the singularity highlights the smoothness of the quasi-histopolant and the absence of the Gibbs phenomenon.}
    \label{plotf5}
\end{figure}
\begin{figure}
    \centering
    \includegraphics[width=0.49\linewidth]{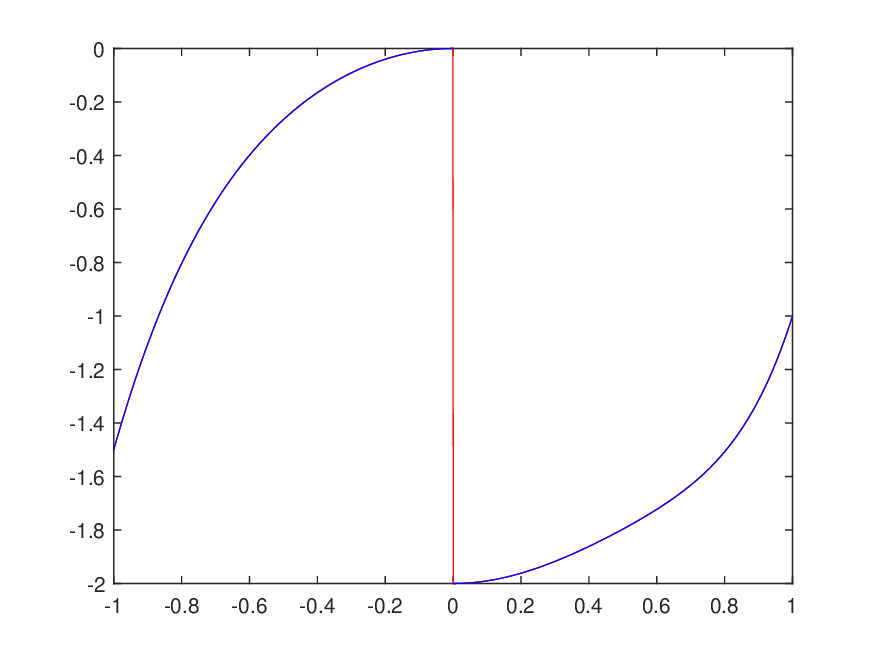}
    \includegraphics[width=0.49\linewidth]{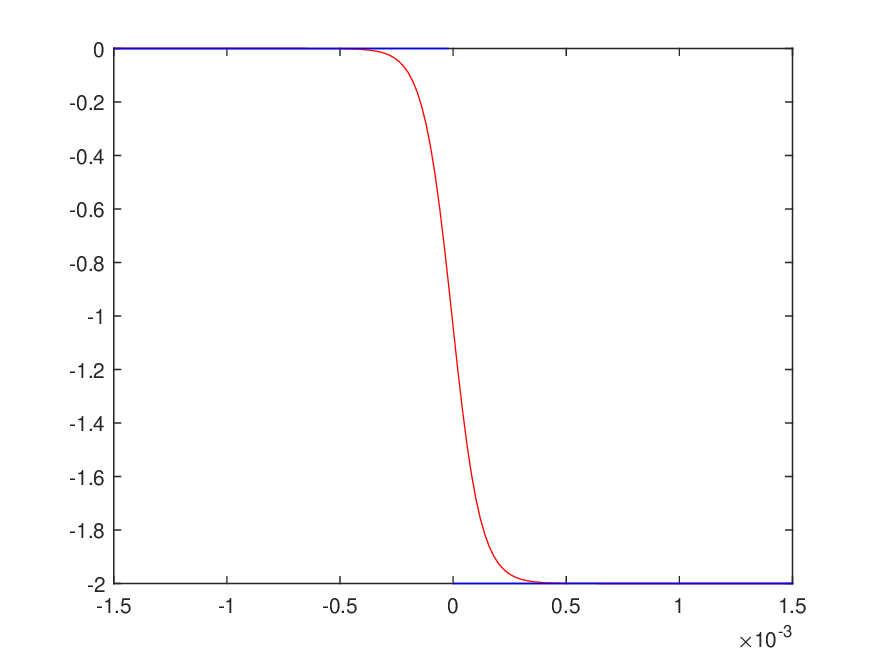}
    \caption{On the left, the plots of the function $f_6$ (blue) and of the quasi-histopolant $\mathcal{Q}_{\mathcal{H},\mu}[f_6]$ (red) computed on the set of segments $\mathcal{S}_n$ based on the set of equispaced nodes $X_n$, with $n=1025$ and by fixing $d=3$. On the right, a zoom near the singularity highlights the smoothness of the quasi-histopolant and the absence of the Gibbs phenomenon.}
    \label{plotf6}
\end{figure}
\begin{figure}
    \centering
    \includegraphics[width=0.49\linewidth]{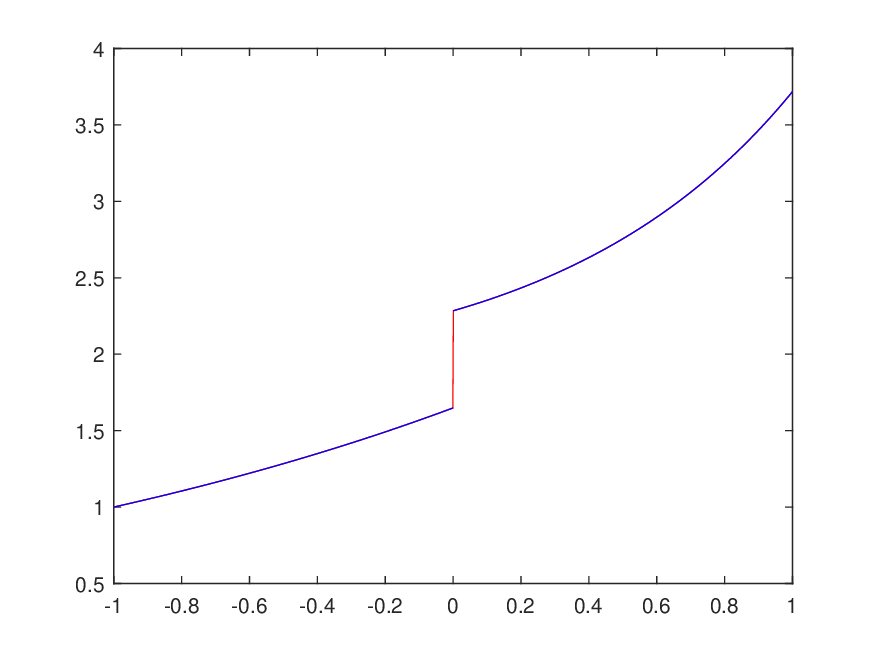}
    \includegraphics[width=0.49\linewidth]{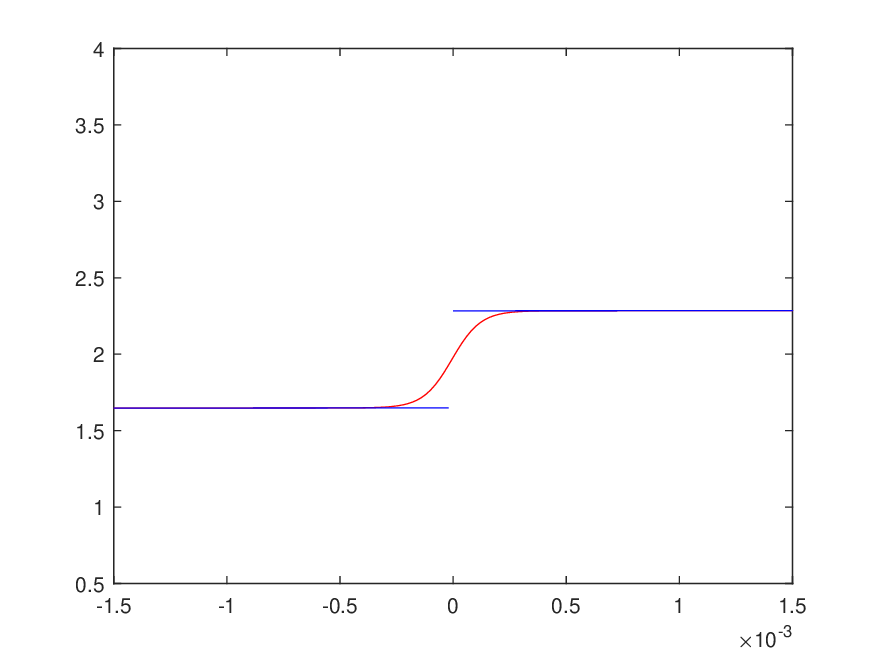}
    \caption{On the left, the plots of the function $f_7$ (blue) and of the quasi-histopolant $\mathcal{Q}_{\mathcal{H},\mu}[f_7]$ (red) computed on the set of segments $\mathcal{S}_n$ based on the set of equispaced nodes $X_n$, with $n=1025$ and by fixing $d=3$. On the right, a zoom near the singularity highlights the smoothness of the quasi-histopolant and the absence of the Gibbs phenomenon.}
    \label{plotf7}
\end{figure}
\begin{figure}
    \centering
    \includegraphics[width=0.49\linewidth]{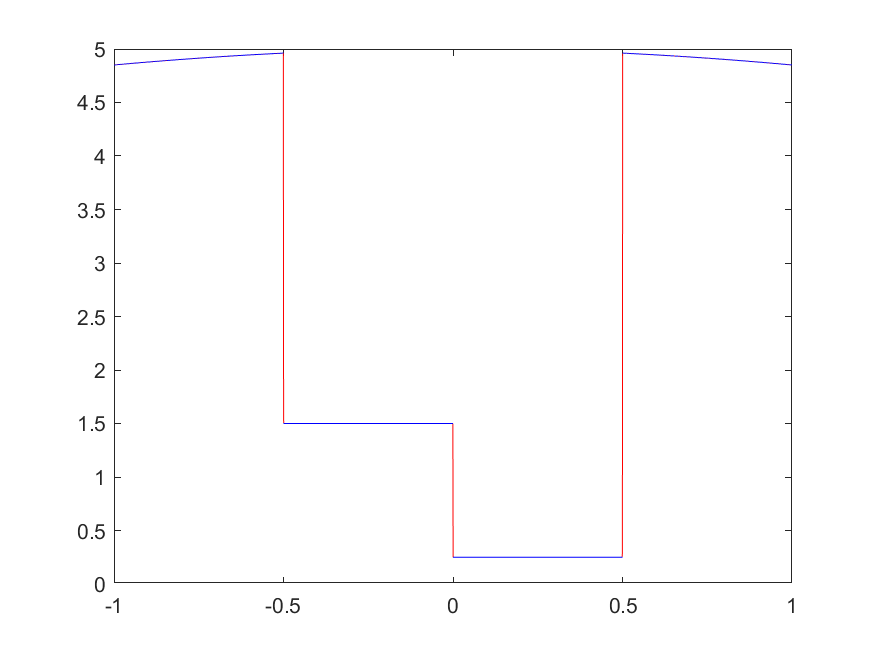}
    \includegraphics[width=0.49\linewidth]{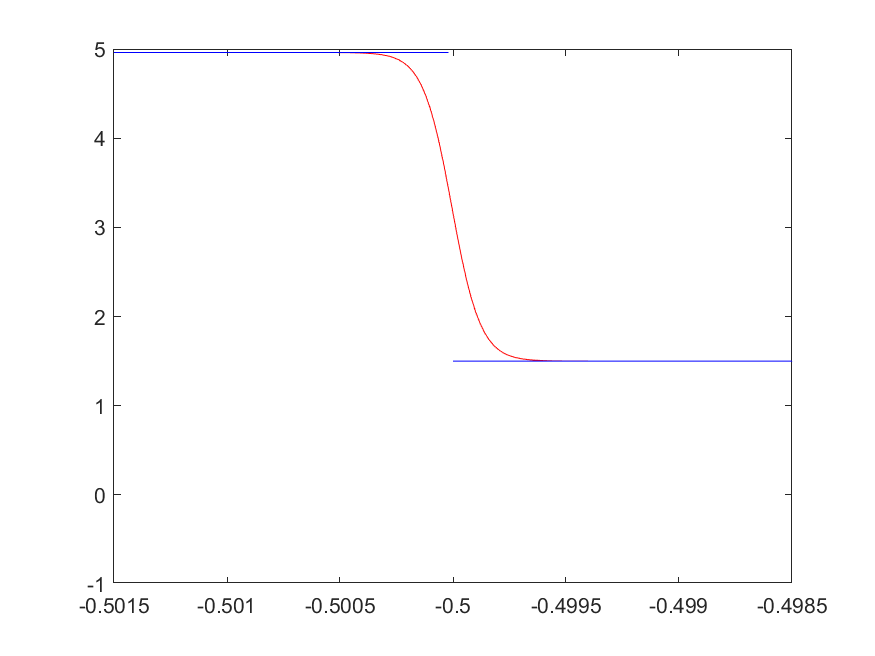}

    \includegraphics[width=0.49\linewidth]{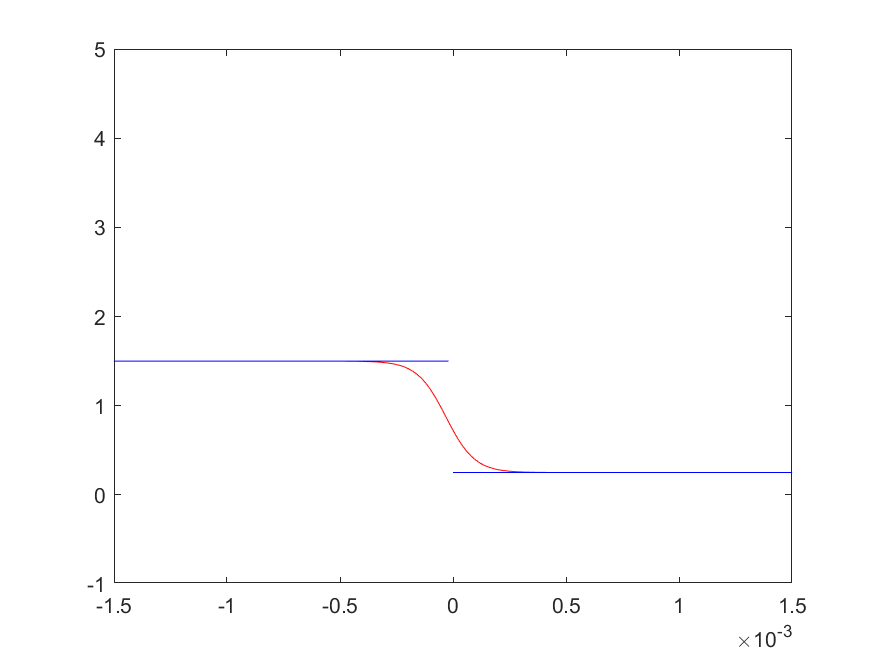}
    \includegraphics[width=0.49\linewidth]{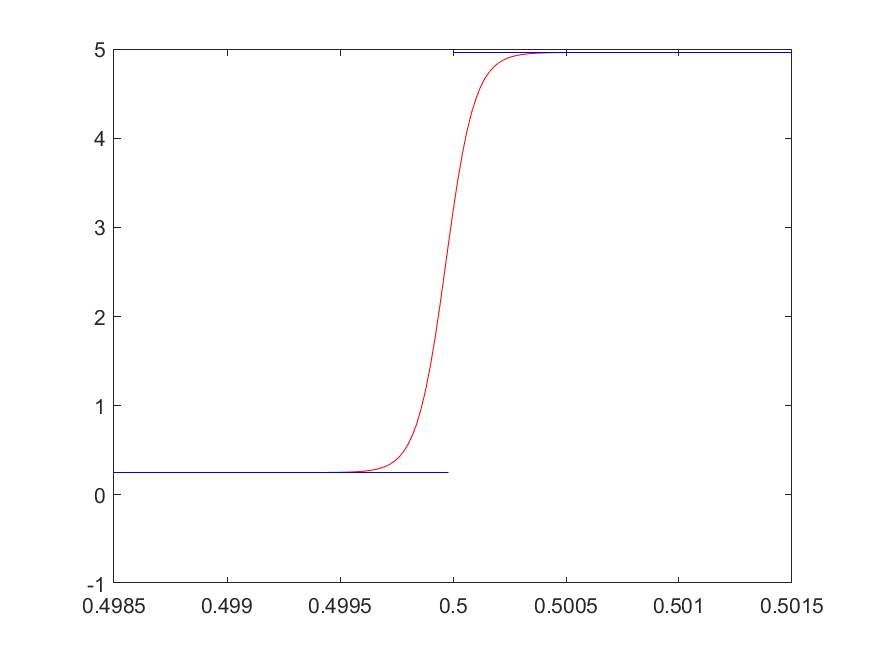}
    \caption{On the left, the plots of the function $f_8$ (blue) and of the quasi-histopolant $\mathcal{Q}_{\mathcal{H},\mu}[f_8]$ (red) computed on the set of segments $\mathcal{S}_n$ based on the set of equispaced nodes $X_n$, with $n=1025$ and by fixing $d=3$. On the right, a zoom near the singularity highlights the smoothness of the quasi-histopolant and the absence of the Gibbs phenomenon.}
    \label{plotf8}
    \end{figure}
    \newline
Finally, for the function $f_5$, we compute the maximum approximation error~\eqref{erroremassimo} on the set of segments $\mathcal{S}_n$ based on the set of equispaced nodes $X_n$, with $n=1025$ and relative to the set of $n_e$ equispaced points, along with the minimum distance between $X_{n_e}$ and the singularity $s=0$ for different values of $n_e$, starting from $500$ and doubling for $i=1, \dots, 4$, with degrees $d=2, 3, 4, 5$. The points that realize the minimum distance between $X_{n_e}$ and $s=0$, for any value of $n_e$, are denoted by $\mp t_i$, as shown in Fig.~\ref{singularitypointwithti}.
The results are presented in Tables~\ref{tab1}--\ref{tab3} for $K=10, 15, 20$, respectively. From these Tables, we observe that as the number $n_e$ increases, the points $\mp t_i$ approach zero. It is noteworthy that for $i=1$, all $n_e$ evaluation points lie within the interval of continuity of the function $f_1$, whereas for $i=2, 3, 4$, some evaluation points lie outside this interval. Consequently, as the number of evaluation points $n_e$ increases, the minimum distance decreases, leading to a deterioration in the maximum approximation error due to the presence of the singularity.
On the other hand, increasing the degree $d$ reduces the maximum approximation error until it reaches a larger value, which depends on $K$.

 \begin{figure}
     \centering
    \includegraphics[width=0.6\linewidth]{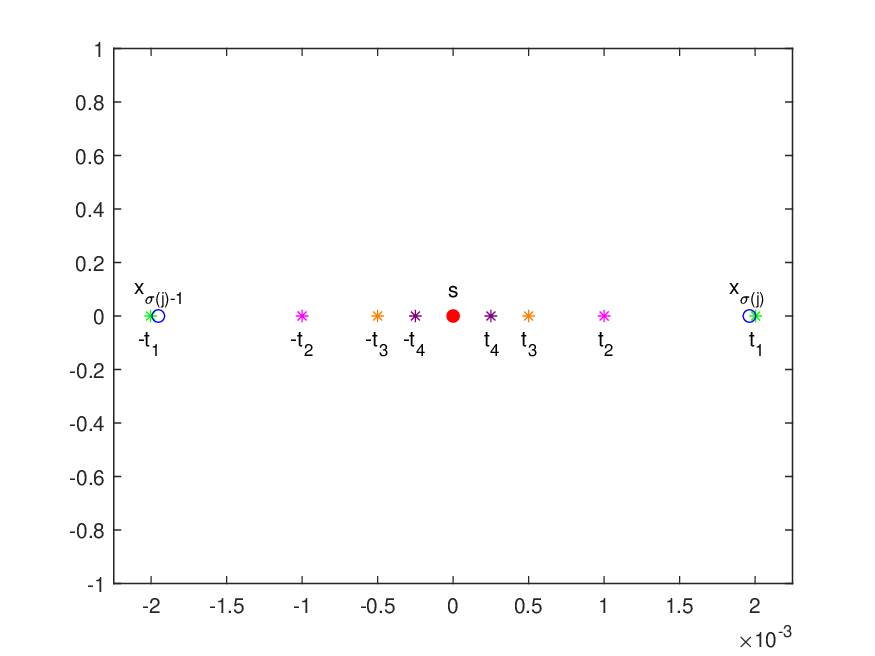}
     \caption{Plot of the singularity point $s=0$ of the function $f_5$, the subsequent nodes $x_{\sigma(j)-1}$, $x_{\sigma(j)}$ bounding an interval containing $s$ and the four couples of two consecutive evaluation points $\mp t_i$, $i=1,\dots,4$, bounding intervals which contain the singularity point $s$ and no any other evaluation point.}
     \label{singularitypointwithti}
 \end{figure}
    \begin{table}
    \centering
    \begin{tabular}{|c|c|c|c|c|c|c|c|}
    \hline
    $i$ & $n_e$ & $\mp t_i$ & $d=2$ & $d=3$ & $d=4$ & $d=5$ \\\hline
    $1$ & $500$ & 2.0040e-03 &  5.1525e-07 &
   4.9831e-09 & 5.8677e-11 & 9.2664e-10  \\
    $2$ & $1000$ & 1.0010e-03 & 2.2003e-06 &
   1.0656e-06 & 1.1993e-05 & 6.0061e-04  \\
    $3$ & $2000$ & 5.0025e-04 & 2.8706e-03 & 3.8893e-03 & 1.5064e-02 & 2.6906e-01\\
    $4$ & $4000$ & 2.5006e-04 & 2.7313e-01 & 2.0466e-01 & 4.6859e-01 & 3.6115e+00  \\\hline        
    \end{tabular}
    \caption{Maximum approximation error for the function $f_5$ computed for $K=10$, $n=1025$ and varying the degree $d$ of the local interpolation polynomials from $2$ to $5$.
     The positions of the points $\mp t_i$ are also shown in column $3$.}
    \label{tab1}
\end{table}

    \begin{table}
    \centering
    \begin{tabular}{|c|c|c|c|c|c|c|c|}
    \hline
    $i$ & $n_e$ & $\mp t_i$ & $d=2$ & $d=3$ & $d=4$ & $d=5$ \\\hline
    $1$ & $500$ & 2.0040e-3 & 5.1525e-07 &  4.8759e-09 &
   5.8677e-11 & 5.8653e-13  \\
    $2$ & $1000$ & 1.0010e-03 & 1.9819e-06 &
   5.6674e-09 & 6.6691e-09 & 3.8303e-07  \\
    $3$ & $2000$ & 5.0025e-04 & 1.7335e-05 & 9.3846e-05 & 3.9392e-04 & 5.7528e-03\\
    $4$ & $4000$ & 2.5006e-04 & 1.5626e-02 & 3.9503e-02 & 8.0300e-02 & 5.9887e-01  \\\hline        
    \end{tabular}
    \caption{Maximum approximation error for the function $f_5$ computed for $K=15$, $n=1025$ and varying the degree $d$ of the local interpolation polynomials from $2$ to $5$.
     The positions of the points $\mp t_i$ are also shown in column $3$.}
    \label{tab2}
\end{table}

\begin{table}
    \centering
    \begin{tabular}{|c|c|c|c|c|c|c|c|}
    \hline
    $i$ & $n_e$ & $\mp t_i$ & $d=2$ & $d=3$ & $d=4$ & $d=5$ \\\hline
    $1$ & $500$ & 2.0040e-03 & 5.1525e-07 &  4.8538e-09 &
   5.8677e-11 & 5.7643e-13  \\
    $2$ & $1000$ & 1.0010e-03 &1.9819e-06 & 5.6576e-09 &
   3.0537e-10 & 2.4306e-10  \\
    $3$ & $2000$ & 5.0025e-04 & 3.5375e-06 & 1.2977e-06 & 1.0058e-05 & 1.2101e-04\\
    $4$ & $4000$ & 2.5006e-04 & 4.5416e-03 & 4.4043e-03 & 1.3102e-02 & 7.2323e-02  \\\hline        
    \end{tabular}
    \caption{Maximum approximation error for the function $f_5$ computed for $K=20$, $n=1025$ and varying the degree $d$ of the local interpolation polynomials from $2$ to $5$.
     The positions of the points $\mp t_i$ are also shown in column $3$.}
    \label{tab3}
\end{table}

\subsection{Numerical test $3$}
    In the third type of numerical tests, we calculate the trend of the approximation error in the $L_1$-norm for the functions $f_i$, $i=5,\dots,8$, based on the set of segments $\mathcal{S}_n$ for $n=100:200:1500$, $d=3$ and $K=10$. The results are shown in Fig.~\ref{trendL1normdisfun}, where we can observe that the error $e_1[f_i]$, $i=5,\dots,8$, decreases by increasing $n$.   

    Since for the functions $f_i$, $i=6,7,8$, the maximum approximation error goes rapidly to the epsilon machine, in Fig.~\ref{trendmaxerror} we only plot the trend of the maximum approximation error~\eqref{erroremassimo} for the function $f_5$ in the interval of continuity.
\begin{figure}
    \centering  \includegraphics[width=0.49\linewidth]{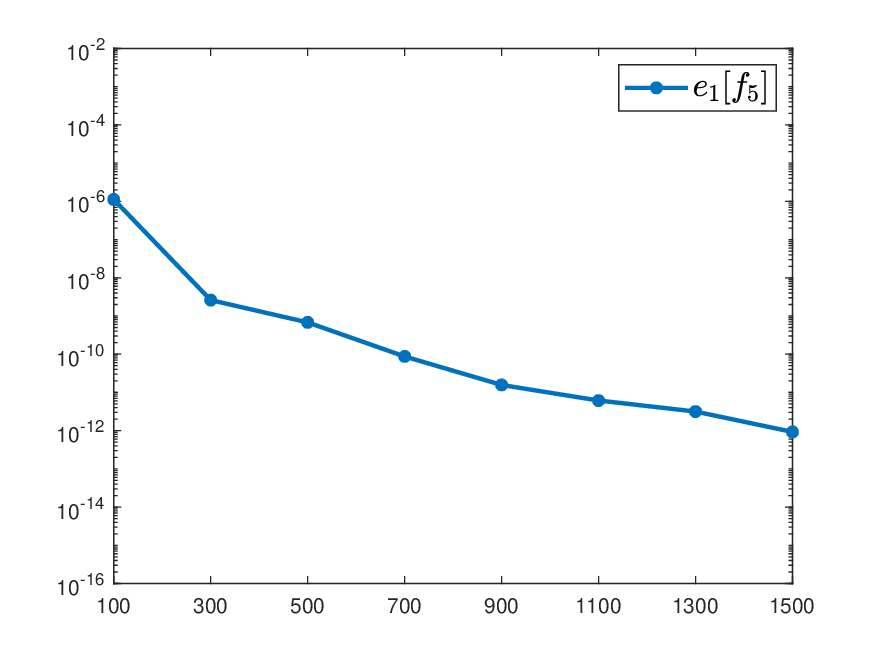}   \includegraphics[width=0.49\linewidth]{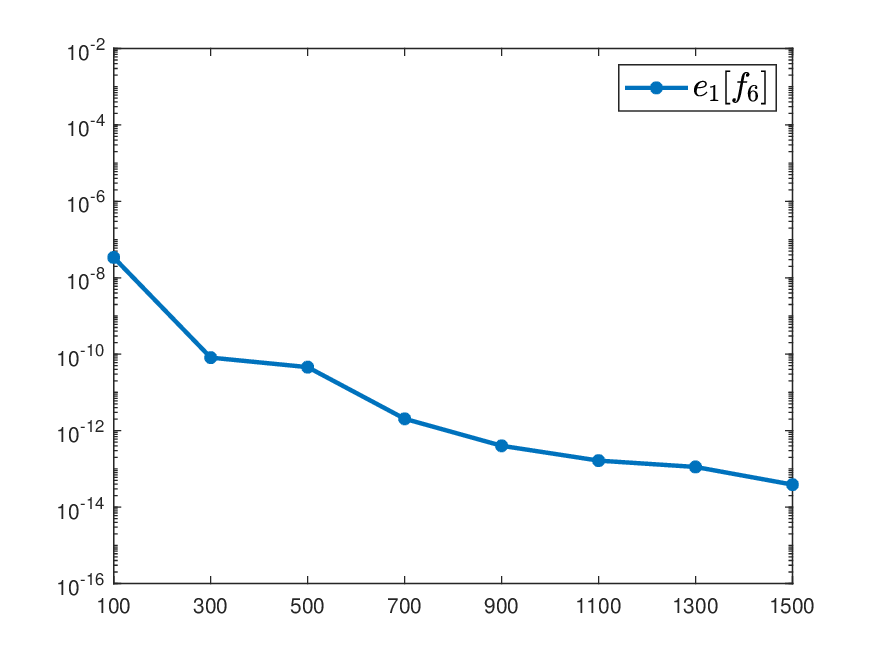}

\includegraphics[width=0.49\linewidth]{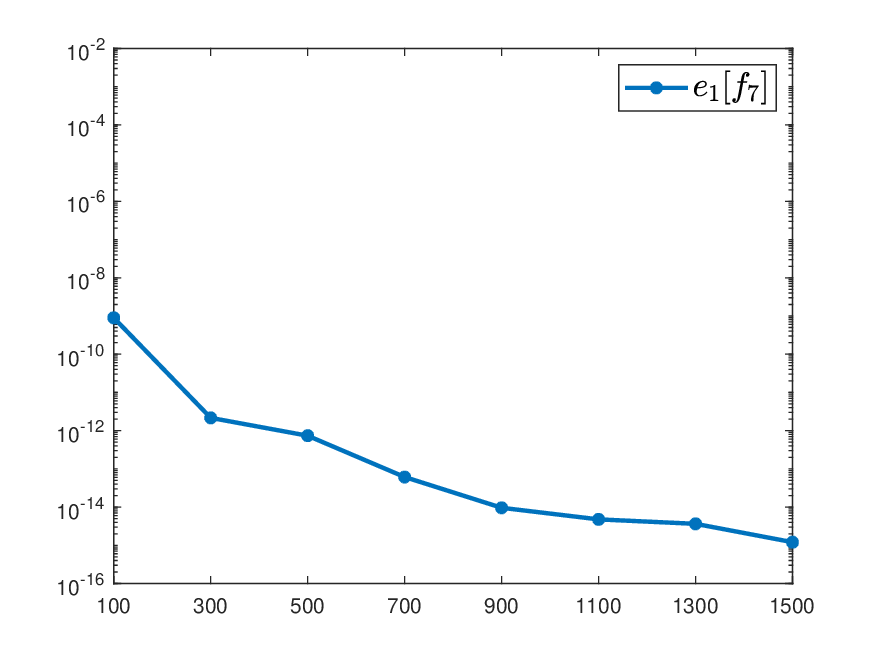}
\includegraphics[width=0.49\linewidth]{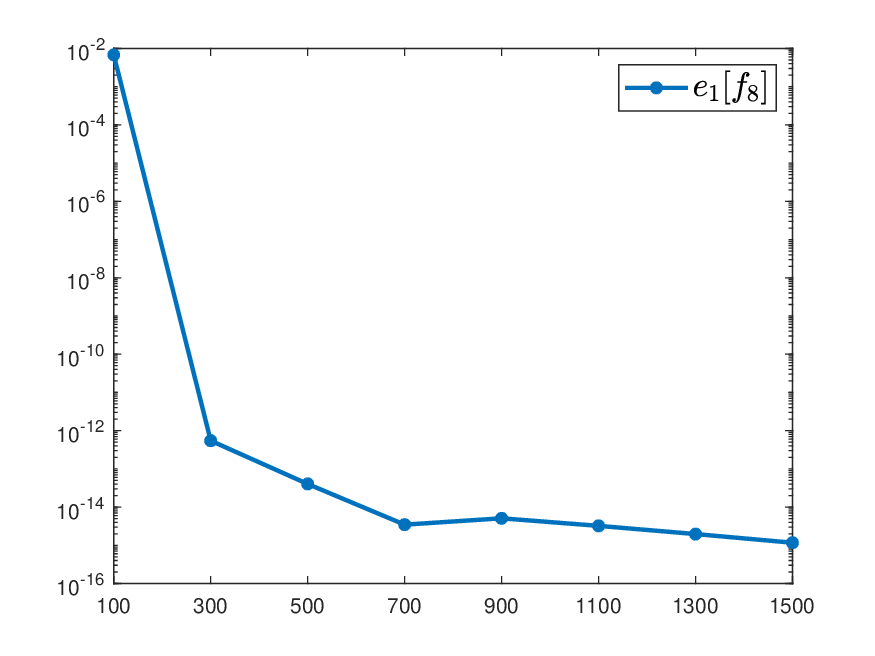}
    \caption{Trend of the approximation error in $L_1$-norm for the functions $f_i$, $i=5,6,7,8$, based on the set of segments $\mathcal{S}_n$, $n=100:200:1500$ and by using $d=3$, $K=10$.}
    \label{trendL1normdisfun}
\end{figure}

\begin{figure}
    \centering
    \includegraphics[width=0.5\linewidth]{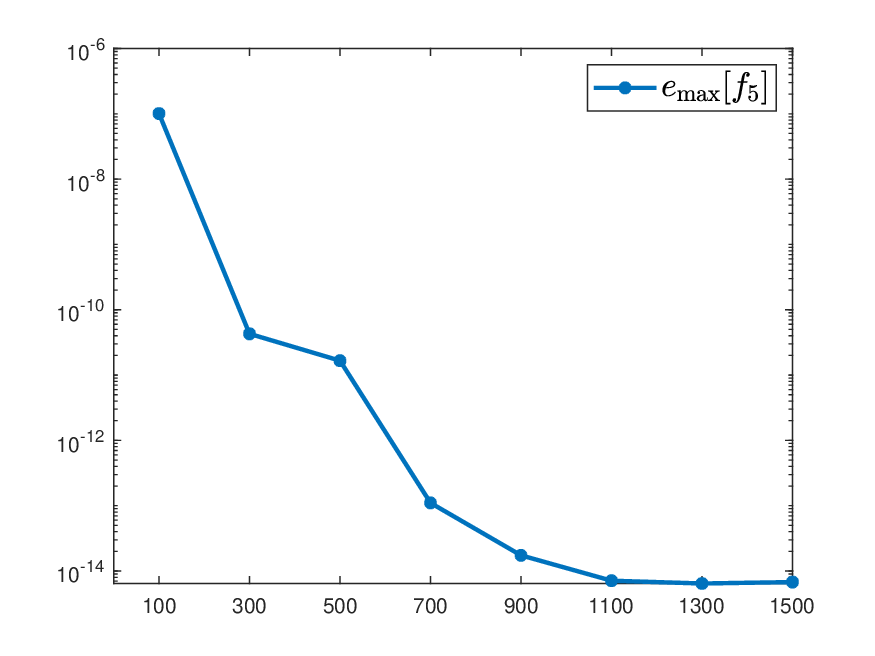}
    \caption{Trend of $e_{\max}[f_5]$ for $K=25$, $d=5$ and for different values of $n=100:200:1500$.} 
    \label{trendmaxerror}
\end{figure}

\subsection{Numerical test $4$}
In the final set of numerical experiments, we consider the test functions, previously used in~\cite{Bruno:2024:PHM}:
\begin{equation*}
\begin{array}{lll}
    g_1(x) = \dfrac{1}{1 + 25x^2}, & g_2(x) = \dfrac{1}{1 + 8x^2}, & g_3(x) = e^{x^{2}+1}, \\
    & & \\
    g_4(x) = \cos(5x), & g_5(x) = \dfrac{1}{x-1.5}, & g_6(x) = x \left\lvert x \right\rvert^3.
    \end{array}
\end{equation*}
In the following, we compare the approximation accuracy achieved by our quasi-histopolation method (for different values of $d$) against standard histopolation on equispaced segments, as well as with the following methods introduced in~\cite{Bruno:2024:PHM}:
\begin{itemize}
    \item Concatenated mock-Chebyshev method ($e^{\mathrm{MC}}$);
    \item Quasi-nodal mock-Chebyshev method ($e^{\mathrm{MCF}}$); 
    \item Constrained mock-Chebyshev segments method ($\hat{e}^{\mathrm{MCF}}$).
\end{itemize}
All tests are performed with $\mu = 4$ and $d = 3:3:12$.

The maximum approximation error is computed on the set of segments $\mathcal{S}_n$ based on the set of equispaced nodes $X_n$, with $n = 51$. The results are reported in Table~\ref{tab:firsttesti}.
As observed, for functions $g_1$, $g_2$, and $g_6$, our quasi-histopolation method already outperforms the other approaches for $d=3$, yielding significantly smaller errors. For the other functions, the approximation error for $d = 3$ is notably higher than that obtained with the constrained mock-Chebyshev segments method. However, as the degree $d$ increases, the error associated with our method decreases and eventually becomes comparable to, or even smaller than, that of the other methods.  
\begin{table}[H]
    \centering
    \begin{tabular}{ c | c | c | c | c | c | c| c| c}
        & $e^{\mathrm{eq}}_n$ & $e^{\mathrm{MC}}$ & $e^{\mathrm{MCF}}$ & $\hat{e}^{\mathrm{MCF}}$ & $d=3$ & $d=6$ & $d=9$ & $d=12$  \\ \hline\hline
        $g_1(x)$ & 4.77e+06 & 6.19e-02 & 7.39e-02 & 2.67e-01 & 2.01e-03 & 5.77e-04 & 3.02e-03 & 2.17e-04 \\ \hline
        $g_2(x)$ & 6.57e+02 & 1.12e-02 & 9.19e-03 & 1.25e-02 & 1.42e-04 & 3.04e-05 & 2.87e-05 & 2.70e-06 \\ \hline
        $g_3(x)$ & 4.05e-03 & 2.10e-08 & 8.48e-10 & 5.90e-13 & 2.48e-05 & 4.77e-07 & 3.52e-10 & 2.90e-12 \\ \hline
        $g_4(x)$ & 2.67e-03 & 9.12e-07 & 2.85e-08 & 7.43e-13 & 4.75e-05 & 1.31e-06 & 4.77e-09 & 6.77e-12 \\ \hline
        $g_5(x)$ & 1.68e-03 & 6.43e-06 & 1.61e-06 & 2.94e-08 & 4.74e-05 & 4.24e-06 & 1.01e-07 & 1.10e-08\\ \hline
        $g_6(x)$ & 1.53e+06 & 1.31e-04 & 1.22e-04 & 2.33e-04 & 5.83e-06 & 6.78e-06 & 1.18e-05 & 2.54e-07 \\ \hline
    \end{tabular}
    \caption{Comparison of the maximum approximation error for the functions $g_i$, $i=1,\dots,6$, on the set of segments $\mathcal{S}_n$ based on the equispaced node set $X_n$ with $n = 51$, obtained using standard histopolation ($e^{\mathrm{eq}}_n$), concatenated mock-Chebyshev ($e^{\mathrm{MC}}$), quasi-nodal mock-Chebyshev ($e^{\mathrm{MCF}}$), constrained mock-Chebyshev segments ($\hat{e}^{\mathrm{MCF}}$), and the proposed quasi-histopolation method for different values of $d$, ranging from $d=3$ to $d=12$.}
    \label{tab:firsttesti}
\end{table}

\section{Conclusions and future works}\label{conclusions}
To approximate functions with jump discontinuities, in this work, we have introduced a smooth, $C^{\infty}$ rational quasi‐histopolation operator by blending local histopolation polynomials on a few nodes using multinode Shepard functions as blending functions. Numerical experiments on both smooth and discontinuous test functions confirm high accuracy in approximation and the
absence of Gibbs-type oscillations.
Future research will focus on extending this approach to higher-dimensional settings and investigating its potential applications in image reconstruction.

\section*{Acknowledgments}
The authors are grateful to the anonymous reviewers for carefully reading the manuscript and for their precise and helpful suggestions which allowed to improve the work.

This research has been achieved as part of RITA \textquotedblleft Research
 ITalian network on Approximation'' and as part of the UMI group ``Teoria dell'Approssimazione
 e Applicazioni''. The research was supported by GNCS-INdAM 2025 project \lq\lq Polinomi, Splines e Funzioni Kernel: dall'Approssimazione Numerical Software Open-Source\rq\rq. The authors are members of the INdAM-GNCS Research group. The work of F. Nudo has been funded by the European Union – NextGenerationEU under the Italian National Recovery and Resilience Plan (PNRR), Mission 4, Component 2, Investment 1.2 \lq\lq Finanziamento di progetti presentati da giovani ricercatori\rq\rq,\ pursuant to MUR Decree No. 47/2025.

\bibliographystyle{elsarticle-num}
\bibliography{bibliography.bib}

\begin{thebibliography}{10}
\expandafter\ifx\csname url\endcsname\relax
  \def\url#1{\texttt{#1}}\fi
\expandafter\ifx\csname urlprefix\endcsname\relax\def\urlprefix{URL }\fi
\expandafter\ifx\csname href\endcsname\relax
  \def\href#1#2{#2} \def\path#1{#1}\fi

\bibitem{Davis:1975:IAA}
P.~J. Davis, Interpolation and {A}pproximation, Dover Publications, 1975.

\bibitem{Wang:2004:QIW}
R.-H. Wang, J.-X. Wang, Quasi-interpolations with interpolation property, J.
  Comput. Appl. Math. 163~(1) (2004) 253--257.

\bibitem{Runge:1901:UEF}
C.~Runge, et~al., {\"U}ber empirische funktionen und die interpolation zwischen
  {\"a}quidistanten ordinaten, Zeitschrift f{\"u}r Mathematik und Physik
  46~(224-243) (1901) 20.

\bibitem{Arandiga1}
F.~Ar{\`a}ndiga, R.~Donat, S.~L{\'o}pez-Ure{\~n}a, Nonlinear improvements of
  quasi-interpolanting splines to approximate piecewise smooth functions, Appl
  Math Comput. 448 (2023) 127946.

\bibitem{Arandiga:2024:EAW}
F.~Ar{\`a}ndiga, D.~Barrera, S.~Eddargani, {ENO} and {WENO} cubic
  quasi-interpolating splines in {B}ernstein--{B}{\'e}zier form, Math. Comput.
  Simul. 225 (2024) 513--527.

\bibitem{Arandiga:2024:NUW}
F.~Ar{\`a}ndiga, D.~Barrera, S.~Eddargani, M.~Ib{\'a}{\~n}ez, J.~Rold{\'a}n,
  Non-uniform {WENO}-based quasi-interpolating splines from the
  {B}ernstein--{B}{\'e}zier representation and applications, Math. Comput.
  Simul. 223 (2024) 158--170.

\bibitem{Arandiga2}
F.~Ar{\`a}ndiga, S.~Remogna, Shape-preserving ${C}^1$ and ${C}^2$
  reconstructions of discontinuous functions using spline quasi-interpolation,
  Mathematics 13 (2025) 1237.

\bibitem{Gibbs:1898:FSS}
J.~W. Gibbs, Fourier's series, Nature 59 (1898) 200--200.

\bibitem{Bruno:2023:PIO}
L.~Bruni~Bruno, W.~Erb, Polynomial interpolation of function averages on
  interval segments, SIAM J. Numer. Anal. 62~(4) (2024) 1759--1781.

\bibitem{Schoenberg:1973:SFA}
I.~J. Schoenberg, Splines and histograms, in: Spline functions and
  approximation theory ({P}roc. {S}ympos., {U}niv. {A}lberta, {E}dmonton,
  {A}lta., 1972), Internat. Ser. Numer. Math., Vol. 21, Birkh\"{a}user Verlag,
  Basel-Stuttgart, 1973, pp. 277--327.

\bibitem{Fischer:2005:MPR}
M.~Fischer, P.~Oja, Monotonicity preserving rational spline histopolation, J.
  Comput. Appl. Math. 175 (2005) 195--208.

\bibitem{Fischer:2007:CSP}
M.~Fischer, P.~Oja, H.~Trossmann, Comonotone shape-preserving spline
  histopolation, J. Comput. Appl. Math. 200 (2007) 127--139.

\bibitem{Siewer:2008:HIS}
R.~Siewer, Histopolating splines, J. Comput. Appl. Math. 220 (2008) 661--673.

\bibitem{Hallik:2017:QLR}
H.~Hallik, P.~Oja, Quadratic/linear rational spline histopolation, BIT Numer.
  Math. 57~(3) (2017) 629--648.

\bibitem{Barnsley:2023:HFF}
M.~F. Barnsley, P.~Viswanathan, Histopolating fractal functions, J. Comput.
  Appl. Math. 425 (2023) 115073.

\bibitem{Hiptmair:2007:NAS}
R.~Hiptmair, J.~Xu, Nodal auxiliary space preconditioning in {${\bf H}({\bf
  curl})$} and {${\bf H}({\rm div})$} spaces, SIAM J. Numer. Anal. 45~(6)
  (2007) 2483--2509.

\bibitem{Bosner:2020:AOC}
T.~Bosner, B.~Crnkovi{\'c}, J.~{\v{S}}kifi{\'c}, Application of
  {CCC}--{S}choenberg operators on image resampling, BIT Numer. Math. 60 (2020)
  129--155.

\bibitem{Petrushev:2011:RAO}
P.~P. Petrushev, V.~A. Popov, Rational approximation of real functions,
  Cambridge University Press, 2011.

\bibitem{DellAccio:2016:APT}
F.~Dell’Accio, F.~Di~Tommaso, K.~Hormann, On the approximation order of
  triangular {S}hepard interpolation, IMA J. Numer. Anal. 36~(1) (2016)
  359--379.

\bibitem{DellAccio:2019:RCM}
F.~Dell'Accio, F.~Di~Tommaso, Rate of convergence of multinode {S}hepard
  operators, Dolomites Res. Notes Approx. 12~(1) (2019) 1--6.

\bibitem{DellAccio:2020:HSM}
F.~Dell'Accio, F.~Di~Tommaso, On the hexagonal {S}hepard method, Appl. Numer.
  Math. 150 (2020) 51--64.

\bibitem{Demichelis:1995:GAO}
V.~Demichelis, Graphic applications of some interpolating weighted mean
  functions, Rocky Mt. J. Mathem. (1995) 1277--1286.

\bibitem{Canny:1986:ACA}
J.~Canny, A {C}omputational {A}pproach to {E}dge {D}etection, IEEE Trans.
  Pattern Anal. 8 (1986) 679–698.

\bibitem{dell2016multinode}
F.~Dell’Accio, F.~Di~Tommaso, K.~Hormann, Multinode rational operators for
  univariate interpolation, AIP Conference Proceedings 1776~(1) (2016).

\bibitem{DellAccio:2018:ROF}
F.~Dell’Accio, F.~Di~Tommaso, K.~Hormann, Reconstruction of a function from
  {H}ermite--{B}irkhoff data, Appl. Math. Comput. 318 (2018) 51--69.

\bibitem{Bruno:2024:PHM}
L.~B. Bruno, F.~Dell'Accio, W.~Erb, F.~Nudo, Polynomial histopolation on
  mock-{C}hebyshev segments, J. Sci. Comput. 104 (2025) 65.

\bibitem{Amat:2022:ACO}
S.~Amat, D.~Levin, J.~Ruiz-Alvarez, J.~C. Trillo, D.~F. Y{\'a}{\~n}ez, A class
  of ${C}^2$ quasi-interpolating splines free of {G}ibbs phenomenon, Numer.
  Algorithms 91 (2022) 51--79.

\bibitem{Costarelli:2017:ADS}
D.~Costarelli, A.~M. Minotti, G.~Vinti, Approximation of discontinuous signals
  by sampling {K}antorovich series, J. Math. Anal. Appl. 450 (2017) 1083--1103.

\end{thebibliography}

\end{document}